\documentclass[11pt,reqno]{amsart}

\usepackage{amsmath,amssymb,amscd,amsxtra,amsfonts,mathrsfs}
\usepackage{epsf,graphicx,epsfig,cite,latexsym,color,hyperref}

\newtheorem{theorem}{Theorem}[section]

\newtheorem{lemma}[theorem]{Lemma}

\theoremstyle{definition}

\theoremstyle{remark}

\newtheorem{example}[theorem]{Example}
\numberwithin{equation}{section}

\begin{document}

\title[An adaptive finite element DtN method]{An
adaptive finite element DtN method for Maxwell's equations}

\author{Gang Bao}
\address{School of Mathematical Sciences, Zhejiang University, Hangzhou 310027,
China.}
\email{baog@zju.edu.cn}

\author{Mingming Zhang}
\address{School of Mathematical Sciences, Zhejiang University, Hangzhou 310027,
China.}
\email{mmzaip@zju.edu.cn}

\author{Xue Jiang}
\address{Faculty of Science, Beijing University of Technology, Beijing,
100124, China.}
\email{jxue@lsec.cc.ac.cn}

\author{Peijun Li}
\address{Department of Mathematics, Purdue University, West Lafayette, Indiana
47907, USA.}
\email{lipeijun@math.purdue.edu}

\author{Xiaokai Yuan}
\address{School of Mathematics, Jilin University, Changchun, Jilin 130012, China.}
\email{yuanxk@jlu.edu.cn}

\thanks{The work of GB is supported in part by an NSFC Innovative Group Fund (No.11621101). The work of XJ is supported partially by NSFC grants 11771057 and 11671052. The research of PL is supported in part by the NSF grant DMS-1912704. The work of XY is supported partially by NSFC grants 12171017 and 12171201. }

\subjclass[2010]{78A45, 78M10, 65N30, 65N12, 65N50}

\keywords{Maxwell's equations, electromagnetic scattering problem, adaptive finite element method, DtN operator, transparent boundary condition, a posteriori error estimate}

\begin{abstract}
This paper is concerned with a numerical solution to the scattering of a time-harmonic electromagnetic wave by a bounded and  impenetrable obstacle in three dimensions. The electromagnetic wave propagation is modeled by a boundary value problem of Maxwell's equations in the exterior domain of the obstacle. Based on the Dirichlet-to-Neumann (DtN) operator, which is defined by an infinite series, an exact transparent boundary condition is introduced and the scattering problem is reduced equivalently into a bounded domain. An a posteriori error estimate based adaptive finite element DtN method is developed to solve the discrete variational problem, where the DtN operator is truncated into a sum of finitely many terms. The a posteriori error estimate takes  into account both the finite element approximation error and the truncation error of the DtN operator. The latter is shown to decay exponentially with respect to the truncation
parameter. Numerical experiments are presented to illustrate the effectiveness of the proposed method.    
\end{abstract}

\maketitle

\section{Introduction}\label{section:introduction}

Scattering problems are concerned with the interaction between an inhomogeneous medium and an incident field. They have played a fundamental role in a wide range of scientific areas such as radar and sonar, non-destructive testing, geophysical exploration, and medical imaging \cite{CK98}. Motivated by significant applications, scattering problems have received great attention in both of the engineering and mathematical communities. A considerable amount of mathematical and numerical results are available for 
the scattering problems of acoustic, elastic, and electromagntic waves. We refer to the monographs \cite{Monk-2003, Nedelec-2001, KH-SP15} on comprehensive accounts of the electromagnetic scattering theory for Maxwell's equations.  

This paper is concerned with a numerical solution of the time-harmonic electromagnetic scattering problem by bounded and impenetrable obstacles in three dimensions. In addition to the large scale computation of the three-dimensional problem, there are two other main challenges: the scattering problem is imposed in an unbounded domain and the solution may have singularity due to the nonsmooth surface of the obstacle. To handle the first issue, the unbounded domain needs to be truncated into a bounded one and an appropriate boundary condition is required to avoid artificial wave reflection; the second difficulty can be resolved by using the adaptive finite element method to balance the accuracy and computational cost. 

One of the most popular methods for domain truncation is the perfectly matched layer (PML) technique, which was proposed by B\'{e}renger to solve the time-domain Maxwell equations \cite{b-jcp94}. The basic idea of PML is to surround the domain of interest by a layer of artificial media which can attenuate outgoing waves. Mathematically, it was proved in \cite{CM-SISC-1998} that when the thickness of the layer is infinity, the PML solution in the domain of interest is the same as the solution of the original scattering problem. However, in practice, the layer needs to be truncated to finite thickness which inevitably
introduces the truncation error. The overall error contains three parts when applying the finite element method to the PML problem: the truncation error of the PML layer, the discretization error in the PML layer, and the discretization error in the domain of interest. It was shown in \cite{BW-SJNA-2005} that the PML truncation error decays exponentially with respect to the thickness of the layer and the PML parameters. As is known, the artificial PML layer is constructed through the complex coordinate stretching
\cite{CW-MOTL-1994}, which makes the PML layer to be an inhomogeneous medium. It is difficult to balance the efficiency and accuracy if a uniform mesh refinement is used. If a thin PML layer is used to reduce the computational cost, then the discretization error is large since the medium is inhomogeneous in the layer; On the contrary, if the discretization error is 
controlled to be small, then a thick PML layer is preferred, which increases the cost. The a posteriori error estimate based adaptive finite element method is effective to handle this issue. The a posteriori error estimates are computable quantities from numerical solutions. They can be used for mesh modification such as refinement or coarsening \cite{V-1996}. The method can control the error and asymptotically optimize the approximation. Moreover, it can effectively deal with the issue that the solution has local singularities in the domain of interest. It is worth mentioning that even though the solution is smooth, the adaptive finite element method is still desirable due to the inhomogeneous medium in the PML layer. We refer to \cite{CC-mc08, CWZ-2007-SIAM, CW-SINUM-2003, jllz-m2na17, JLLZ-cms18} for the discussion of adaptive finite element PML methods for scattering problems in different structures.

Another effective approach is to impose transparent boundary conditions to solve the scattering problems formulated in open domains. A key step of the method is to construct the Dirichlet-to-Neumann (DtN) operator, which can be done via different manners such as the boundary integral equation \cite{MK-jcp04}, the Fourier transform or Fourier series expansions \cite{LWZ-SIMA-2011, pwz-MMAS-2012,JLLZ-JSC-2017}. In this paper, observing that the solution is analytical when it is away from the obstacle, we consider the Fourier series expansion of the solution on any sphere that encloses the obstacle. The DtN operator can be obtained by studying the resulting systems of ordinary differential equations for the Fourier coefficients. Compared to the PML technique, the DtN method does not introduce an auxiliary layer of inhomogeneous medium, which can reduce the cost. However, the DtN operator is nonlocal and is defined as an infinite series. In actual computation, the infinite series needs to be truncated into a sum of finitely many terms, which also introduces a truncation error. It was shown in \cite{B-SJNA-1995,HNPX-JCAM-2011 } that if the solution is smooth enough, the DtN operator truncation error decays exponentially with respect to the truncation number. When the solution has singularities, the convergence analysis is sophisticated. The DtN operator truncation error needs to be integrated into the a posteriori error estimate and the truncation number can be determined automatically through the estimate. The adaptive finite element DtN method has  been successfully applied to solve many scattering problems, including acoustic waves \cite{BZHL-2020-DCDSS, JLLZ-JSC-2017, JLZ-CCP-2013, LZZ-CSAM20}, electromagnetic waves \cite{JLLWWZ,YBL-CSIAM-2020}, and elastic waves 
\cite{LY-2020-CMAME,BLY-SINUM21,LY-2019}.

This work is a non-trivial extension of the adaptive finite element DtN method for the acoustic and elastic wave scattering problems by bounded obstacles. Compared to the acoustic and elastic scattering problems, the electromagentic scattering problem is more involved. Computationally, it is also more intense to solve the three-dimensional Maxwell equations. In this paper, we deduce an a posteriori error estimate which takes into account both the finite element discretization error and the DtN operator truncation error. The latter is shown to decay exponentially with respect to the truncation number. One of the key steps in the analysis is to consider a new dual problem and to deduce its analytical solution. Based on the a posteriori error estimate, we develop an adaptive finite element DtN method. Numerical experiments are presented to demonstrate the competitive behavior of the proposed method. This work provides a viable alternative to the adaptive finite element PML method for solving the electromagnetic scattering problem. In addition, the adaptive finite element DtN method may be applied to solve many other electromagnetic scattering problems imposed in unbounded domains.

The paper is organized as follows. In Section \ref{section:ProblemF}, we introduce the model problem and some function spaces used in the analysis. The DtN operator and the variational problems are discussed in Section \ref{section:VariationalF}. Section 
\ref{section:PosterioriE} presents the finite element discretization with the truncated DtN operator and states the a posteriori error estimate. Section \ref{section:Proof} is devoted to the proof of the error estimate and is the main part of the work. Numerical experiments are presented in Section \ref{sec:num} to demonstrate the efficiency of the proposed method. The paper is concluded with some general remarks in Section \ref{section: conclusion}.

\section{Problem formulation}\label{section:ProblemF}

Denote by $D$ the domain of the obstacle with Lipschitz boundary $\partial D$. The obstacle is assumed to be contained in the ball $B_{R}=\left\{\boldsymbol{x}\in\mathbb{R}^3: |\boldsymbol{x}|<R\right\}$ with boundary $\Gamma_R=\{\boldsymbol x\in\mathbb R^3: |\boldsymbol x|=R\}$. Let $B_{R'}$ be the smallest ball centered at the origin with radius $R'$ that also contains $\overline{D}$, i.e., $D\subset\subset B_{R'}\subset\subset B_R$ with $0<R'<R$. Denote by $\Omega:=B_R\setminus\overline{D}$ the bounded domain enclosed by $\Gamma_R$ and $\partial D$. The exterior domain $\mathbb R^3\setminus\overline D$ is assumed to be filled with a homogeneous medium characterized by the dielectric permittivity $\epsilon$ and the magnetic permeability $\mu$. Without loss of generality, we may assume that the dielectric permittivity $\epsilon=1$ and the magnetic permeability $\mu=1$. Furthermore, we assume that the obstacle is a perfect electric conductor.

Let the obstacle be illuminated by a time-harmonic electromagnetic field $(\boldsymbol E^{\rm inc}, \boldsymbol H^{\rm inc})$, which can be either a plane wave or a point source. The total electromagnetic field $(\boldsymbol E, \boldsymbol H)$ is governed by Maxwell's equations
\begin{equation}\label{MaxwellEqns}
\begin{cases}
\nabla\times\boldsymbol{E}-{\rm i}\kappa \boldsymbol{H}=0,\quad \nabla\times\boldsymbol{H}+{\rm i}\kappa \boldsymbol{E}=0 & {\rm in} ~ \mathbb{R}^3\setminus\overline{D},\\
\boldsymbol{\nu}\times\boldsymbol{E}=0 &{\rm on} ~ \partial D,\\
|\boldsymbol{x}|\left(\boldsymbol{E}^s-\boldsymbol{H}^s\times\hat{\boldsymbol{x}}\right)\rightarrow 0
&{\rm as} ~ |\boldsymbol{x}|\rightarrow +\infty,
\end{cases}
\end{equation} 
where $\boldsymbol{\nu}$ is the unit normal vector to $\partial D$ pointing to the exterior of $D$, and 
$\boldsymbol{E}^s=\boldsymbol{E}-\boldsymbol{E}^{\rm inc}$ and $\boldsymbol{H}^s=\boldsymbol{H}-\boldsymbol{H}^{\rm inc}$
are the scattered electric and magnetic fields, respectively. Eliminating the magnetic field $\boldsymbol{H}$ from \eqref{MaxwellEqns}, we obtain the Maxwell system for the electric field $\boldsymbol E$: 
\begin{equation}\label{ElectroFields}
\begin{cases}
\nabla\times(\nabla\times\boldsymbol{E})-\kappa^2\boldsymbol{E}=0
\quad & {\rm in} ~ \mathbb{R}^3\setminus\overline{D},\\
\boldsymbol{\nu}\times\boldsymbol{E}=0 \quad &{\rm on} ~ \partial D,\\
|\boldsymbol{x}|\left[\left(\nabla\times\boldsymbol{E}^s\right)\times\hat{\boldsymbol{x}}
-{\rm i}\kappa\boldsymbol{E}^s\right]\rightarrow 0
\quad &{\rm as} ~ |\boldsymbol{x}|\rightarrow +\infty.
\end{cases}
\end{equation} 

Next we introduce some function spaces. Denote by $L^2(\Omega)$ and $\boldsymbol{L}^2(\Omega)=L^2(\Omega)^3$ the standard Hilbert space of complex square integrable functions in $\Omega$ and the corresponding Cartesian product space, respectively. Let
\[
\boldsymbol{H}({\rm curl},\, \Omega):=\left\{\boldsymbol{\phi}\in\boldsymbol{L}^2(\Omega): \nabla\times\boldsymbol{\phi}\in\boldsymbol{L}^2(\Omega)\right\},
\]
which has the norm
\begin{equation}\label{curlnorm}
\|\boldsymbol{\phi}\|_{\boldsymbol{H}({\rm curl},\, \Omega)}=\left(\|\boldsymbol{\phi}\|^2_{\boldsymbol{L}^2(\Omega)}
+\|\nabla\times\boldsymbol{\phi}\|_{\boldsymbol{L}^2(\Omega)}^2\right)^{1/2}.
\end{equation}

To describe the Calder\'{o}n operator and the TBC, it is necessary to introduce some trace function spaces defined on $\Gamma_R$. Let $H^s(\Gamma_R), s\in\mathbb R$ be the standard trace Sobolev space and $\boldsymbol{H}^s(\Gamma_R)=H^s(\Gamma_R)^3$ be the corresponding Cartesian product space. Define the tangential function spaces
\[
TL(\Gamma_R):=\left\{\boldsymbol{\phi}\in\boldsymbol{L}^2(\Gamma_R): \boldsymbol{\phi}\cdot \boldsymbol{e}_{\rho}=0\right\},
\quad TH^{s}(\Gamma_R)=\left\{\boldsymbol{\phi}\in \boldsymbol{H}^s(\Gamma_R): \boldsymbol{\phi}\cdot\boldsymbol{e}_{\rho}=0\right\},
\]
where $\boldsymbol e_\rho$ is the unit normal vector to $\Gamma_R$.  It is shown in \cite[Theorem 6.23]{CK98} that for any $\boldsymbol{\phi}\in TL(\Gamma_R)$, it has the Fourier series expansion
\begin{eqnarray*}
\boldsymbol{\phi}=\sum_{n\in\mathbb{N}}\sum_{|m|\leq n}\phi_{1n}^m\boldsymbol{U}_n^m + \phi_{2n}^m\boldsymbol{V}_n^m,
\end{eqnarray*}
where $\{(\boldsymbol{U}_n^m, \boldsymbol{V}_n^m): |m|\leq n, n=0,1,\dots\} $ is an orthonormal basis for $TL(\Gamma_R)$ (cf. \eqref{Unm}--\eqref{Vnm}). The norm for functions in $TL(\Gamma_R)$ and $TH^s(\Gamma_R)$ can be characterized by 
\[
\|\boldsymbol{\phi}\|_{TL(\Gamma_R)}=\left(\sum_{n\in\mathbb{N}}\sum_{|m|\leq
n}|\phi_{1n}^m|^2+|\phi_{2n}^m|^2\right)^{1/2}
\]
and
\begin{equation*}\label{ths}
\|\boldsymbol{\phi}\|_{TH^s(\Gamma_R)}=\left[\sum_{n\in\mathbb{N}}\sum_{|m|\leq
n}\left(1+n(n+1)\right)^s\left(|\phi_{1n}^m|^2+|\phi_{2n}^m|^2\right)\right]^{1/2}. 
\end{equation*}

Denote by ${\rm curl}_{\Gamma_R}$ and ${\rm div}_{\Gamma_R}$ the surface curl and the surface divergence on $\Gamma_R$ (cf. Appendix \ref{Appendix:SurOpe}), respectively. Let
\begin{align*}
TH^{-1/2}({\rm curl}, \Gamma_R)&=\left\{\boldsymbol{\phi}\in TH^{-1/2}(\Gamma_R): {\rm 
curl}_{\Gamma_R}\boldsymbol{\phi}\in H^{-1/2}(\Gamma_R) \right\},\\
TH^{-1/2}({\rm div}, \Gamma_R)&=\left\{\boldsymbol{\phi} \in TH^{-1/2}(\Gamma_R): {\rm
div}_{\Gamma_R}\boldsymbol{\phi}\in H^{-1/2}(\Gamma_R) \right\}, 
\end{align*}
which are equipped with the norms 
\begin{equation}\label{tcurl}
\|\boldsymbol{\phi}\|_{TH^{-1/2}({\rm curl},\Gamma_R)} =\left[\sum_{n\in\mathbb{N}}\sum_{|m|\leq n} \left(\frac{1}{\sqrt{1+n(n+1)}}|\phi_{1n}^m|^2+\sqrt{1+n(n+1)}|\phi_{2n}^m|^2\right)\right]^{1/2}
\end{equation}
and
\begin{equation*}\label{tdiv}
\|\boldsymbol{\phi}\|_{TH^{-1/2}({\rm div},\Gamma_R)} =\left[\sum_{n\in\mathbb{N}}\sum_{|m|\leq n}\left( \sqrt{1+n(n+1)}|\phi_{1n}^m|^2+\frac{1}{\sqrt{1+n(n+1)}}|\phi_{2n}^m|^2\right)\right]^{1/2}.
\end{equation*}

Hereafter, the notation $a\lesssim b$ stands for $a\leq Cb$, where $C$ is a positive constant whose value is not required and may change step by step in the proofs.

\section{Variational problems}\label{section:VariationalF}

In this section, we introduce a TBC to reduce the boundary value problem \eqref{ElectroFields} into the bounded domain $\Omega$ and discuss its variational formulation.

Given a tangential vector $\boldsymbol{\phi}$ on $\Gamma_R$, it has the Fourier series expansion
\begin{equation}\label{TangentialV}
	\boldsymbol{\phi}=\sum_{n\in\mathbb{N}}\sum_{|m|\leq n}\phi_{1n}^m\boldsymbol{U}_n^m
	+ \phi_{2n}^m\boldsymbol{V}_n^m. 
\end{equation}
The Calder\'{o}n operator  $\mathscr{T}: \boldsymbol{H}^{-1/2}({\rm curl}, \Gamma_R)\rightarrow\boldsymbol{H}^{-1/2}(\rm div, \Gamma_R)$ is defined by 
\begin{equation}\label{CapacityOpe}
\mathscr{T}\boldsymbol{\phi}=\sum_{n\in\mathbb{N}}\sum_{|m|\leq n}
\frac{{\rm i}\kappa R }{1+z_n^{(1)}(\kappa R)}  \phi_{1n}^m\boldsymbol{U}_n^m
+ \frac{1+z_n^{(1)}(\kappa R) }{{\rm i}\kappa R}\phi_{2n}^m\boldsymbol{V}_n^m ,
\end{equation}
where
$ z_n^{(1)}(z)=t h_n^{(1)'}(z)/h_n^{(1)}(z) $ and $h_n^{(1)}(z)$ is the spherical Hankel function of the first kind
with order $n$. It is shown in \cite{BW-SJNA-2005} that the solution $\boldsymbol{E}$ of \eqref{ElectroFields}
satisfies the following TBC on $\Gamma_R$: 
\begin{equation}\label{TBC}
\left(\nabla\times\boldsymbol{E}\right)\times\boldsymbol{\nu}-{\rm i}\kappa \mathscr{T}\boldsymbol{E}_{\Gamma_R}=\boldsymbol{f},
\end{equation}
where $\boldsymbol{E}_{\Gamma_R}=\boldsymbol{e}_{\rho}\times(\boldsymbol{E}\times\boldsymbol{e}_{\rho})$ is the tangential component of $\boldsymbol E$ and 
\[
\boldsymbol{f}=(\nabla\times\boldsymbol{E}^{\rm inc})\times\boldsymbol{\nu}-{\rm i}\kappa \mathscr{T}\boldsymbol{E}^{\rm inc}_{\Gamma_R}. 
\]

Based on \eqref{TBC}, the boundary value problem \eqref{ElectroFields} can be equivalently reduced into the bounded domain $\Omega$. The corresponding variational problem is to find 
$\boldsymbol{E}\in \boldsymbol{H}_{\partial D}({\rm curl},\Omega)$  such that
\begin{equation}\label{vform}
a(\boldsymbol{E},\boldsymbol{\psi})=\int_{\Gamma_R} \boldsymbol{f}\cdot\overline{\boldsymbol{\psi}}_{\Gamma_R}{\rm d}s
\quad \forall \boldsymbol{\psi}\in\boldsymbol{H}_{\partial D}({\rm curl},\Omega),
\end{equation}
where the sesquilinear form 
$a: \boldsymbol{H}(\rm curl, \Omega)\times\boldsymbol{H}(\rm curl, \Omega)\rightarrow\mathbb{C}$ is defined by
\begin{equation}\label{auv}
a(\boldsymbol{\varphi},\boldsymbol{\psi})=\int_{\Omega}(\nabla\times\boldsymbol{\varphi})\cdot(\nabla\times
\overline{\boldsymbol{\psi}}){\rm d}\boldsymbol{x} -
\kappa^2\int_{\Omega}\boldsymbol{\varphi}\cdot\overline{\boldsymbol{\psi}}{\rm d}\boldsymbol{x}
-{\rm i}\kappa\int_{\Gamma_R}\mathscr{T}\boldsymbol{\varphi}_{\Gamma_R}\cdot
\overline{\boldsymbol{\psi}}_{\Gamma_R}{\rm d}s
\end{equation}
and
$
\boldsymbol{H}_{\partial D}({\rm curl}, \Omega)=\left\{\boldsymbol{\phi}\in\boldsymbol{H}({\rm curl}, \Omega):
\boldsymbol{\nu}\times\boldsymbol{\phi}=0\,{\rm on}\,\partial D\right\}.
$

The well-posedness of the variational problem \eqref{vform} is discussed in \cite{Monk-2003, Nedelec-2001}. Here we simply assume that the variational problem \eqref{vform} has a unique weak solution $\boldsymbol E\in \boldsymbol H_{\partial D}({\rm curl},\Omega)$. By the general theory in Babuska and Aziz \cite{BA-AP-1973}, there exists a constant $\gamma>0$ depending on $\kappa$ and $R$ such that the following inf-sup condition holds: 
\begin{equation}\label{inf-sup}
\sup_{0\neq \boldsymbol{\psi}\in \boldsymbol{H}_{\partial D}({\rm curl},\Omega)}
	\frac{|a(\boldsymbol{\varphi},\boldsymbol{\psi})|}{\|\boldsymbol{\psi}\|_{\boldsymbol{H}({\rm curl},\Omega)}}\geq
	\gamma \|\boldsymbol{\varphi}\|_{\boldsymbol{H}({\rm curl},\Omega)}\quad\forall\boldsymbol{\varphi}\in 
	\boldsymbol{H}_{\partial D}({\rm curl},\Omega).
\end{equation}

In practice, the Calder\'{o}n operator needs to be truncated into a sum of finitely many terms. Define the truncated Calder\'{o}n operator 
\begin{equation}\label{TrunCapacity}
\mathscr{T}^N\boldsymbol{\phi}=\sum_{n\leq N}\sum_{|m|\leq n} \frac{{\rm i}\kappa R }{1+z_n^{(1)}(\kappa R)}  \phi_{1n}^m\boldsymbol{U}_n^m+ \frac{1+z_n^{(1)}(\kappa R) }{{\rm i}\kappa R}\phi_{2n}^m\boldsymbol{V}_n^m ,
\end{equation}
where $N>0$ is a sufficiently large integer and $\boldsymbol{\phi}$ is a tangent vector on $\Gamma_R$ with the Fourier series expansion \eqref{TangentialV}.

Replacing $\mathscr{T}$ in \eqref{vform} by $\mathscr T^N$, we obtain the truncated variational problem which is to find  $\boldsymbol{E}^N\in \boldsymbol{H}_{\partial D}({\rm curl},\Omega)$ 
such that
\begin{equation}\label{vNform}
a^N(\boldsymbol{E}^N,\boldsymbol{\psi})=\int_{\Gamma_R} \boldsymbol{f}^N\cdot\overline{\boldsymbol{\psi}}_{\Gamma_R}{\rm d}s
\quad \forall \boldsymbol{\psi}\in\boldsymbol{H}_{\partial D}({\rm curl},\Omega),
\end{equation}
where the sesquilinear form 
$a^N: \boldsymbol{H}(\rm curl, \Omega)\times\boldsymbol{H}(\rm curl, \Omega)\rightarrow\mathbb{C}$ is defined by
\begin{equation}\label{aNuv}
a^N(\boldsymbol{\varphi},\boldsymbol{\psi})=\int_{\Omega}(\nabla\times\boldsymbol{\varphi})\cdot(\nabla\times
\overline{\boldsymbol{\psi}}){\rm d}\boldsymbol{x} -
\kappa^2\int_{\Omega}\boldsymbol{\varphi}\cdot\overline{\boldsymbol{\psi}}{\rm d}\boldsymbol{x}
-{\rm i}\kappa\int_{\Gamma_R}\mathscr{T}^N\boldsymbol{\varphi}_{\Gamma_R}\cdot
\overline{\boldsymbol{\psi}}_{\Gamma_R}{\rm d}s
\end{equation}
and
\begin{equation*}\label{fn}
\boldsymbol{f}^N:=(\nabla\times\boldsymbol{E}^{\rm inc})\times\boldsymbol{\nu}-{\rm i}\kappa \mathscr{T}^N\boldsymbol{E}^{\rm inc}_{\Gamma_R}.
\end{equation*}

Following the argument in \cite{HNPX-JCAM-2011}, we may show that the truncated variational problem \eqref{vNform} admits a unique weak solution $\boldsymbol{E}^N\in\boldsymbol{H}_{\partial D}({\rm curl},\Omega)$ for the sufficiently large truncation number $N$. The details are omitted since the focus of this work is on the a posteriori error estimate of the solution to the truncated variational problem.

\section{The a posteriori error estimate}\label{section:PosterioriE}

In this section, we introduce the finite element approximation of variational problem \eqref{vNform}, present the a posteriori error estimate, and state the main result of this paper.

Let $\mathcal{M}_h$ be a regular tetrahedral mesh of the domain $\Omega$, where $h$ denotes the maximum diameter of all the elements in $\mathcal M_h$. For any tetrahedral element which has more than one vertex on  $\Gamma_R$, we may adopt the mapping $T$ proposed in \cite{Bernardi-1989-SINUM} which maps the corresponding edge or face of the tetrahedral element exactly on $\Gamma_R$.

Introduce the lowest order N\'ed\'elec edge element space with isoparametric mapping 
\[
\boldsymbol{V}_h:=\left\{\boldsymbol{v}(T^{-1}(\boldsymbol{x})): \boldsymbol{v}\in \boldsymbol{H}_{\partial D}({\rm curl}, \Omega), \boldsymbol v|_{K}=\boldsymbol{a}_K+\boldsymbol{b}_K\times\boldsymbol{x} ~ \forall K\in\mathcal{M}_h\right\}.
\]
The finite element approximation to the variational problem \eqref{vNform} is to find $\boldsymbol{E}^N_h\in\boldsymbol{V}_h$ such that
\begin{equation}\label{vNhform}
a_h^N(\boldsymbol{E}_h^N,\boldsymbol{\psi}_h)=\int_{\Gamma_R} \boldsymbol{f}^N\cdot
	\overline{\boldsymbol{\psi}}_h\quad \forall \boldsymbol{\psi}_h\in\boldsymbol{V}_h,
\end{equation}
where the sesquilinear form 
$a_h^N: \boldsymbol{V}_h\times\boldsymbol{V}_h\rightarrow\mathbb{C}$ is defined by
\begin{equation}\label{aNh}
a_h^N(\boldsymbol{\varphi},\boldsymbol{\psi})=\int_{\Omega}(\nabla\times\boldsymbol{\varphi})\cdot(\nabla\times
\overline{\boldsymbol{\psi}}){\rm d}\boldsymbol{x} -
\kappa^2\int_{\Omega}\boldsymbol{\varphi}\cdot\overline{\boldsymbol{\psi}}{\rm d}\boldsymbol{x}
-{\rm i}\kappa\int_{\Gamma_R}\mathscr{T}^N\boldsymbol{\varphi}_{\Gamma_R}\cdot
\overline{\boldsymbol{\psi}}_{\Gamma_R}{\rm d}s.
\end{equation}

For sufficiently small $h$, the discrete inf-sup condition of the sesquilinear form \eqref{aNuv} may be
established by following the approach in \cite{Schatz-MC-1974}. Based on the general theory in \cite{BA-AP-1973},
the discrete variational problem \eqref{vNhform} has a unique solution  $\boldsymbol E^N_h\in\boldsymbol{V}_h$. Again, the details
of the proof are omitted since our focus is the a posteriori error estimate.

Next, we introduce the a posteriori error estimate in the tetrahedral elements and across their faces.
For any tetrahedral element $K\in\mathcal{M}_h$, we define the local residuals by
\begin{eqnarray*}
R_K^{(1)}&:=&\kappa^2\boldsymbol{E}_h^N|_K-
\nabla\times(\nabla\times\boldsymbol{E}_h^N)|_K,\\
R_K^{(2)}&:=&-\kappa^2\nabla\cdot\boldsymbol{E}_h^N|_K.
\end{eqnarray*}
Denote by $\mathcal{F}_h$ the set of all faces of tetrahedrons in $\mathcal{M}_h$.
Given an interior face $F\in\mathcal{F}_h$, which is the common face of elements
$K_1$ and $K_2$, we define the jump residuals across $F$ as
\begin{eqnarray*}
J_F^{(1)}&:=&\left(\nabla\times\boldsymbol{E}_h^N|_{K_1}-
\nabla\times\boldsymbol{E}_h^N|_{K_2}\right)\times\boldsymbol{\nu},\\
J_F^{(2)}&:=&\kappa^2\left(\boldsymbol{E}_h^N|_{K_1}-
\boldsymbol{E}_h^N|_{K_2}\right)\cdot\boldsymbol{\nu},
\end{eqnarray*}
where the unit normal vector $\boldsymbol{\nu}$ on $F$ points from $K_2$
to $K_1$. Given a face $F\in\mathcal{F}_h\cap\Gamma_R$, 
define the residuals by
\begin{eqnarray*}
J_F^{(1)}&:=&2\left[-\left(\nabla\times\boldsymbol{E}_h^N\right)\times\boldsymbol{\nu}_1
+{\rm i}\kappa \mathscr{T}^N\left(\boldsymbol{E_h^N}\right)_{\Gamma_R}+\boldsymbol{f}^N\right],\\
J_F^{(2)}&:=&2\left[\kappa^2\boldsymbol{E}_h^N\cdot\boldsymbol{\nu}_1
-{\rm i}\kappa {\rm div}_{\Gamma_R}(\mathscr{T}^N(\boldsymbol{E_h^N})_{\Gamma_R})
-{\rm div}_{\Gamma_R} \boldsymbol{f}^N\right].
\end{eqnarray*}
For any $K\in\mathcal{M}_h$, denote by $\eta_K$ the local error
estimator, which is defined by
\begin{equation*}\label{eta_T}
\eta_K^2=h_K^2\left(\|R_K^{(1)}\|^2_{\boldsymbol{L}^2(K)}+\|R_K^{(2)}\|^2_{\boldsymbol{L}^2(K)}\right)
+h_K\sum_{F\in\partial K}\left(
\|J_F^{(1)}\|^2_{\boldsymbol{L}^2(F)}+\|J_F^{(2)}\|^2_{\boldsymbol{L}^2(F)}\right) .  
\end{equation*}

The main result of the paper is stated as follows.

\begin{theorem}\label{thm}
Let $\boldsymbol{E}$ and $\boldsymbol{E}_h^N$ be the solutions of \eqref{vform} and \eqref{vNhform}, respectively.
Then for sufficiently large integer $N$, the following a posteriori error estimate holds:
	\[
	\|\boldsymbol{E}-\boldsymbol{E}_h^N\|_{\boldsymbol{H}({\rm curl}, \Omega)}\lesssim
	\Bigg(\sum_{K\in\mathcal{M}_h}\eta_K^2\Bigg)^{1/2}+
	\left(\frac{R'}{R}\right)^{N+1}\|\boldsymbol{f}\|_{TH^{-1/2}({\rm div}, \Gamma_R)}.
	\]
\end{theorem}

It is clear to note from the theorem that there are two parts for the a posteriori error estimate. One comes from the finite element discretization error and another takes into account the truncation error of the DtN operator. Moreover, the latter decreases exponentially with respect to $N$ since $R'<R$. 

\section{Proof of the main theorem}\label{section:Proof}

This section is devoted to the proof of Theorem \ref{thm}. Denote the error by $\boldsymbol{\xi}=\boldsymbol{E}-\boldsymbol{E}_h^N$. A simple calculation from \eqref{auv} shows that
\begin{eqnarray}\label{Key}
	\|\boldsymbol{\xi}\|^2_{\boldsymbol{H}(\rm curl, \Omega)}&=&\Re\left\{a(\boldsymbol{\xi}, \boldsymbol{\xi})
	+{\rm i}\kappa\int_{\Gamma_R} \left(\mathscr{T}-\mathscr{T}^N\right)\boldsymbol{\xi}_{\Gamma_R}
	\cdot\overline{\boldsymbol{\xi}}_{\Gamma_R}{\rm d}s\right\} \notag\\
&&	-\kappa\Im\left\{\int_{\Gamma_R} \mathscr{T}^N\boldsymbol{\xi}_{\Gamma_R}\cdot\overline{\boldsymbol{\xi}}_{\Gamma_R}
	{\rm d}s\right\}+\left(\kappa^2+1\right)\int_{\Omega}\boldsymbol{\xi}\cdot\overline{\boldsymbol{\xi}}{\rm d}\boldsymbol{x}.
\end{eqnarray}
It suffices to estimate the four terms on the right-hand side of \eqref{Key}. In the following, Lemmas \ref{ExponentialDecay} and \ref{Key2} concern the estimates of the first two terms; Lemma \ref{imtbc} is devoted to the estimate of the third term; Lemmas \ref{w2n}--\ref{Lemma59} address the error estimate of the last term. 

Let us begin with some trace regularity results. Similar results can be found in \cite[Lemmas 3.3 and 3.4]{pwz-MMAS-2012} for the overfilled cavity problem of Maxwell's equations. 

\begin{lemma}\label{TraceRegularity}
For any $\boldsymbol{\phi}\in\boldsymbol{H}({\rm curl}, \Omega)$, the following estimate holds:
\[
\|\boldsymbol{\phi}_{\Gamma_R}\|_{TH^{-1/2}({\rm curl}, \Gamma_R)}\leq C\|\boldsymbol{\phi}\|_{\boldsymbol{H}({\rm curl}, \Omega)},
\]
where $C>0$ only depends on the domain $\Omega$.
\end{lemma}

\begin{proof}
Let $\Omega':=B_{R}\setminus \overline{B_{R'}}\subset\Omega$. For any $\boldsymbol{\phi}\in\boldsymbol{H}({\rm curl}, \Omega)$, it has the following Fourier series expansion in $\Omega'$: 
\begin{eqnarray*}
\boldsymbol{\phi} =\sum\limits_{n\in\mathbb{N}}\sum\limits_{|m|\leq n}\phi_{1n}^m(\rho)\boldsymbol{U}_n^m+\phi_{2n}^m(\rho)\boldsymbol{V}_n^m+\phi_{3n}^m(\rho)X_n^m\boldsymbol{e}_{\rho}.
\end{eqnarray*}
It is easy to verify from \eqref{curlUVX} that
\begin{eqnarray}\label{curlphi}
\nabla\times\boldsymbol{\phi}& =& -\frac{1}{\rho}\sum\limits_{n\in\mathbb{N}}\sum\limits_{|m|\leq n} \sqrt{n(n+1)}\phi_{2n}^m(\rho)X_{n}^m \boldsymbol{e}_{\rho}-\frac{1}{\rho}\sum\limits_{n\in\mathbb{N}}\sum\limits_{|m|\leq n} \frac{\rm d}{{\rm d}\rho}\left(\rho\phi_{2n}^m(\rho)\right)\boldsymbol{U}_n^m \notag\\
&& -\frac{1}{\rho}	\sum\limits_{n\in\mathbb{N}}\sum\limits_{|m|\leq n} \Big(\sqrt{n(n+1)}\phi_{3n}^m(\rho)-\frac{\rm d}{{\rm d}\rho}\left(\rho\phi_{1n}^m(\rho)\right)\Big)\boldsymbol{V}_n^m.
\end{eqnarray}
Substituting \eqref{curlphi} into \eqref{curlnorm}, we obtain 
\begin{eqnarray*}
&&\|\boldsymbol{\phi}\|^2_{\boldsymbol{H}({\rm curl}, B_R\setminus\overline{B_R'})}
=\frac{1}{R^2}\sum\limits_{n\in\mathbb{N}}\sum\limits_{|m|\leq n} \Bigg\{
\int_{R'}^{R} \left[\rho^2+n(n+1)\right] |\phi_{2n}^m(\rho)|^2
+\left|\frac{\rm d}{{\rm d}\rho}\left(\rho\phi_{2n}^m(\rho)\right)\right|^2{\rm d}\rho\\
&&\quad+\int_{R'}^{R}\left|\sqrt{n(n+1)}\phi_{3n}^m(\rho)-\frac{\rm d}{{\rm d}\rho}\left(\rho\phi_{1n}^m(\rho)\right)\right|^2{\rm d}\rho +\int_{R'}^R \rho^2\left(|\phi_{1n}^m|^2+|\phi_{3n}^m|^2\right){\rm d}\rho
\Bigg\}\\
&&=\frac{1}{R^2}\sum\limits_{n\in\mathbb{N}}\sum\limits_{|m|\leq n} \Bigg\{
\int_{R'}^{R} \left[\rho^2+n(n+1)\right] |\phi_{2n}^m(\rho)|^2
+\left|\frac{\rm d}{{\rm d}\rho}\left(\rho\phi_{2n}^m(\rho)\right)\right|^2{\rm d}\rho\\
&&\quad+\int_{R'}^R\rho^2\left|\phi_{1n}^m\right|^2
+\frac{\rho^2}{\rho^2+n(n+1)}\left|\frac{\rm d}{{\rm d}\rho}\left(\rho\phi_{1n}^m(\rho)\right)\right|^2{\rm d}\rho\\
&&\quad+\int_{R'}^R
\left|\sqrt{\frac{n(n+1)}{\rho^2+n(n+1)}}\frac{\rm d}{{\rm d}\rho}\left(\rho\phi_{1n}^m(\rho)\right)
-\sqrt{\rho^2+n(n+1)}\phi_{3n}^m(\rho)\right|^2 {\rm d}\rho\Bigg\}, 
\end{eqnarray*}
which gives
\begin{eqnarray}\label{tracestability1}
\|\boldsymbol{\phi}\|^2_{\boldsymbol{H}({\rm curl}, \Omega)}&\geq& 
\|\boldsymbol{\phi}\|^2_{\boldsymbol{H}({\rm curl}, B_R\setminus\overline{B_R'})}\notag\\
&\geq&\frac{1}{R^2}\sum\limits_{n\in\mathbb{N}}\sum\limits_{|m|\leq n} \Bigg\{
\int_{R'}^{R} \left[\rho^2+n(n+1)\right] |\phi_{2n}^m(\rho)|^2
+\left|\frac{\rm d}{{\rm d}\rho}\left(\rho\phi_{2n}^m(\rho)\right)\right|^2{\rm d}\rho\notag\\
&& +\int_{R'}^R\rho^2\left|\phi_{1n}^m\right|^2
+\frac{\rho^2}{\rho^2+n(n+1)}\left|\frac{\rm d}{{\rm d}\rho}\left(\rho\phi_{1n}^m(\rho)\right)\right|^2{\rm d}\rho
\Bigg\}.
\end{eqnarray}

A simple calculation shows that
\begin{eqnarray}\label{Trace2}
&& (R-R')R^2\left|\xi(R)\right|^2\notag\\
&=& \int_{R'}^R\left|\rho\xi(\rho)\right|^2{\rm d}\rho
+\int_{R'}^R\int_{\rho}^R \frac{\rm d}{{\rm d}\tau}\left|\tau\xi(\tau)\right|^2{\rm d}\tau{\rm d}\rho\notag\\
&\leq&\int_{R'}^R\left|\rho\xi(\rho)\right|^2{\rm d}\rho
+2(R-R')\int_{R'}^R\left|\rho\xi(\rho)\right|\left|\frac{\rm d}{{\rm d}\rho}\left(\rho\xi(\rho)\right)\right|{\rm d}\rho.
\end{eqnarray}
It follows from \eqref{Trace2} that
\begin{eqnarray}\label{phi1nm}
&&\frac{1}{\sqrt{1+n(n+1)}}\left|\phi_{1n}^m(R)\right|^2\notag\\
&\leq&\frac{1}{(R-R')R^2}\frac{1}{\sqrt{1+n(n+1)}}\Bigg[
\int_{R'}^R\left|\rho\phi_{1n}^m(\rho)\right|^2
+(R-R')\sqrt{1+n(n+1)}\left|\rho\phi_{1n}^m(\rho)\right|^2{\rm d}\rho \notag\\
&&+\frac{(R-R')}{\sqrt{1+n(n+1)}}\int_{R'}^R\left|\frac{\rm d}{{\rm d}\rho}
\left(\rho\phi_{1n}^m(\rho)\right)\right|^2{\rm d}\rho
\Bigg]\notag\\
&\leq& \frac{1}{R^2}\left(1+\frac{1}{R-R'}\right)\int_{R'}^R\left|\rho\phi_{1n}^m(\rho)\right|^2{\rm d}\rho
+\frac{1}{1+n(n+1)}\frac{1}{R^2}\int_{R'}^R\left|\frac{\rm d}{{\rm d}\rho}
\left(\rho\phi_{1n}^m(\rho)\right)\right|^2{\rm d}\rho.
\end{eqnarray}
Similarly, we have
\begin{eqnarray}\label{phi2nm}
&&\sqrt{1+n(n+1)}R^2\left|\phi_{2n}^m(R)\right|^2\notag\\
&&\leq \left[1+n(n+1)+\frac{\sqrt{1+n(n+1)}}{R-R'}\right]
\int_{R'}^R\left|\rho\phi_{2n}^m(\rho)\right|^2{\rm d}\rho
+\int_{R'}^R\left|\frac{\rm d}{{\rm d}\rho}
\left(\rho\phi_{2n}^m(\rho)\right)\right|^2{\rm d}\rho.
\end{eqnarray}
Taking the summation of \eqref{phi1nm} and \eqref{phi2nm} over $n\in\mathbb{N}$, and using the definition
\eqref{tcurl} and \eqref{tracestability1}, we complete the proof. 
\end{proof}

\begin{lemma}\label{stability2}
For any $\delta>0$, there is a positive constant $C(\delta)$ such that the following estimate holds:
\[
\|\boldsymbol{\phi}_{\Gamma_R}\|^2_{TH^{-1/2}(\Gamma_R)}
\lesssim\delta\|\nabla\times\boldsymbol{\phi}\|^2_{\boldsymbol{L}^2(\Omega)}
+C(\delta)\|\boldsymbol{\phi}\|_{\boldsymbol{L}^2(\Omega)}^2
\quad\forall\boldsymbol{\phi}\in\boldsymbol{H}(\rm curl, \Omega).
\]
\end{lemma}

\begin{proof}
For any $\boldsymbol{\phi}\in\boldsymbol{H}({\rm curl}, \Omega)$, it has the Fourier series expansion
in $\Omega'=B_{R}\setminus\overline{B_{R'}}$:
\begin{eqnarray*}
\boldsymbol{\phi}=\sum\limits_{n\in\mathbb{N}}\sum\limits_{|m|\leq n}\phi_{1n}^m(\rho)\boldsymbol{U}_n^m+\phi_{2n}^m(\rho)\boldsymbol{V}_n^m+\phi_{3n}^m(\rho)X_n^m\boldsymbol{e}_{\rho}.
\end{eqnarray*}
By \eqref{curlphi}, we have 
	\begin{eqnarray}\label{semicurl}
		&&\|\nabla\times\boldsymbol{\phi}\|^2_{\boldsymbol{L}^2(B_R\setminus\overline{B_{R'}})}
		=\frac{1}{R^2}\sum\limits_{n\in\mathbb{N}}\sum\limits_{|m|\leq n} \Bigg\{
		\int_{R'}^{R} n(n+1) |\phi_{2n}^m(\rho)|^2
		+\left|\frac{\rm d}{{\rm d}\rho}\left(\rho\phi_{2n}^m(\rho)\right)\right|^2{\rm d}\rho \notag\\
		&&\quad+\int_{R'}^{R}\left|\sqrt{n(n+1)}\phi_{3n}^m(\rho)
			-\frac{\rm d}{{\rm d}\rho}\left(\rho\phi_{1n}^m(\rho)\right)\right|^2{\rm d}\rho
		\Bigg\}\\
		&&=\frac{1}{R^2}\sum\limits_{n\in\mathbb{N}}\sum\limits_{|m|\leq n} \Bigg\{
		\int_{R'}^{R} n(n+1)|\phi_{2n}^m(\rho)|^2
		+\left|\frac{\rm d}{{\rm d}\rho}\left(\rho\phi_{2n}^m(\rho)\right)\right|^2{\rm d}\rho\notag\\
		&&\quad+\int_{R'}^R
		\frac{\rho^2}{\rho^2+n(n+1)}\left|\frac{\rm d}{{\rm d}\rho}\left(\rho\phi_{1n}^m(\rho)\right)\right|^2
		-\rho^2\left|\phi_{3n}^m\right|^2{\rm d}\rho\notag\\
		&&\quad+\int_{R'}^R
		\left|\sqrt{\frac{n(n+1)}{\rho^2+n(n+1)}}\frac{\rm d}{{\rm d}\rho}\left(\rho\phi_{1n}^m(\rho)\right)
		-\sqrt{\rho^2+n(n+1)}\phi_{3n}^m(\rho)\right|^2 {\rm d}\rho
		\Bigg\}, \notag
	\end{eqnarray}
	which gives 
	\begin{eqnarray}\label{tracestability2}
		\|\nabla\times\boldsymbol{\phi}\|^2_{\boldsymbol{L}^2(\Omega)}&\geq& 
		\|\nabla\times\boldsymbol{\phi}\|^2_{\boldsymbol{L}^2(B_R\setminus\overline{B_{R'}})}\notag\\
		&\geq&\frac{1}{R^2}\sum\limits_{n\in\mathbb{N}}\sum\limits_{|m|\leq n} \Bigg\{
		\int_{R'}^{R} n(n+1) |\phi_{2n}^m(\rho)|^2
		+\left|\frac{\rm d}{{\rm d}\rho}\left(\rho\phi_{2n}^m(\rho)\right)\right|^2{\rm d}\rho\notag\\
		&&+\int_{R'}^R
		\frac{\rho^2}{\rho^2+n(n+1)}\left|\frac{\rm d}{{\rm d}\rho}\left(\rho\phi_{1n}^m(\rho)\right)\right|^2
		-\rho^2\left|\phi_{3n}^m\right|^2{\rm d}\rho
		\Bigg\}.
	\end{eqnarray}

It follows from \eqref{Trace2} that
	\begin{eqnarray}\label{phi01nm}
	&&\frac{1}{\sqrt{1+n(n+1)}}\left|\phi_{1n}^m(R)\right|^2\notag\\
	&\leq&\frac{1}{(R-R')R^2}\frac{1}{\sqrt{1+n(n+1)}}\Bigg[
	\int_{R'}^R\left|\rho\phi_{1n}^m(\rho)\right|^2
	+\frac{1}{\hat{\delta}}(R-R')\sqrt{1+n(n+1)}\left|\rho\phi_{1n}^m(\rho)\right|^2 {\rm d}\rho \notag\\
	&&+\frac{(R-R')}{\sqrt{1+n(n+1)}}\hat{\delta}\int_{R'}^R\left|\frac{\rm d}{{\rm d}\rho}
		\left(\rho\phi_{1n}^m(\rho)\right)\right|^2{\rm d}\rho
	\Bigg]\notag\\
	&\leq& \frac{1}{R^2}\left(\frac{1}{\hat{\delta}}+\frac{1}{R-R'}\right)\int_{R'}^R\left|\rho\phi_{1n}^m(\rho)\right|^2{\rm d}\rho
	+\frac{1}{1+n(n+1)}\frac{\hat{\delta}}{R^2}\int_{R'}^R\left|\frac{\rm d}{{\rm d}\rho}
		\left(\rho\phi_{1n}^m(\rho)\right)\right|^2{\rm d}\rho. 
	\end{eqnarray}
Similarly, we have 
\begin{eqnarray}\label{phi02nm}
\frac{1}{\sqrt{1+n(n+1)}}\left|\phi_{2n}^m(R)\right|^2
&\leq& \frac{1}{R-R'}\frac{1}{R^2}\frac{1}{\sqrt{1+n(n+1)}}\Bigg[
\int_{R'}^R\left|\rho\phi_{2n}^m(\rho)\right|^2{\rm d}\rho\notag\\
&& +\frac{1}{\hat{\delta}}(R-R')\frac{1}{\sqrt{1+n(n+1)}}\int_{R'}^R\left|\rho\phi_{2n}^m(\rho)\right|^2{\rm d}\rho
\notag\\
&&+\hat{\delta}(R-R')\sqrt{1+n(n+1)}\int_{R'}^R\left|\frac{\rm d}{{\rm d}\rho}\left(\rho\phi_{2n}^m(\rho)\right)\right|^2{\rm d}\rho\Bigg]\notag\\
&\leq& \frac{\hat{\delta}}{R^2}\int_{R'}^R n(n+1)\left|\phi_{2n}^m(\rho)\right|^2
+\left|\frac{\rm d}{{\rm d}\rho}\left(\rho\phi_{2n}^m(\rho)\right)\right|^2{\rm d}\rho\notag\\
&&+\frac{1}{R-R'}\left[\frac{1}{R^2}+\frac{(R-R')^2}{\hat{\delta}}\right]
\int_{R'}^R\left|\rho\phi_{2n}^m(\rho)\right|^2{\rm d}\rho. \notag
\end{eqnarray}
Taking $\hat{\delta}=\delta R^2$ and summing over $n\in\mathbb{N}$, we obtain from \eqref{tracestability2}--\eqref{phi02nm} that
\begin{eqnarray*}
\|\boldsymbol{\phi}_{\Gamma_R}\|_{TH^{-1/2}(\Gamma_R)}^2 
&\leq&\delta\sum\limits_{n\in\mathbb{N}}\sum\limits_{|m|\leq n}\int_{R'}^R\Bigg[
\frac{1}{1+n(n+1)}\left|\frac{\rm d}{{\rm d}\rho}
\left(\rho\phi_{1n}^m(\rho)\right)\right|^2\notag\\
&&+n(n+1)\left|\phi_{2n}^m(\rho)\right|^2
+\left|\frac{\rm d}{{\rm d}\rho}\left(\rho\phi_{2n}^m(\rho)\right)\right|^2
-\left|\rho\phi_{3n}^m(\rho)\right|^2
\Bigg]{\rm d}\rho\\
&&+\delta\sum\limits_{n\in\mathbb{N}}\sum\limits_{|m|\leq n}\int_{R'}^R
\left|\rho\phi_{3n}^m(\rho)\right|^2{\rm d}\rho\notag\\
&&+C(\delta)\sum\limits_{n\in\mathbb{N}}\sum\limits_{|m|\leq n}\int_{R'}^R
\left[\left|\rho\phi_{1n}^m(\rho)\right|^2+
\left|\rho\phi_{2n}^m(\rho)\right|^2+\left|\rho\phi_{3n}^m(\rho)\right|^2\right]{\rm d}\rho\\
&\lesssim & \delta\|\nabla\times\boldsymbol{\phi}\|_{\boldsymbol{L}^2(\Omega)}^2
+C(\delta)\|\boldsymbol{\phi}\|^2_{\boldsymbol{L}^2(\Omega)},
\end{eqnarray*}
which completes the proof.
\end{proof}

The following lemma shows that the Fourier coefficients of the scattered field $\boldsymbol E^s$ decay exponentially with 
respect to the truncation number $N$. The result is crucial to deduce the estimates of the first two terms in \eqref{Key}.

\begin{lemma}\label{ExponentialDecay}
Let the scattered field $\boldsymbol{E}^s=\boldsymbol{E}-\boldsymbol{E}^{\rm inc}$ admit the Fourier series expansion in 
the domain $\mathbb{R}^3\setminus \overline {B_{R'}}$:
\[
\boldsymbol{E}^s=\sum\limits_{n\in\mathbb{N}}\sum\limits_{|m|\leq n}E_{1n}^{sm}(\rho) \boldsymbol{U}_n^m+E_{2n}^{sm}(\rho)\boldsymbol{V}_n^m+E_{3n}^{sm}(\rho)X_{n}^m \boldsymbol{e}_{\rho}.
\]
Then for sufficiently large $N$, the following estimates hold:
\[
\left|E_{jn}^{sm}(R)\right|\lesssim \left(\frac{R'}{R}\right)^n\left|E_{jn}^{sm}(R')\right|,\quad j=1,2.
\]
\end{lemma}

\begin{proof}
It follows from \cite[Theorem 6.27]{CK98}  that
\begin{equation}\label{ERR}
\frac{E_{1n}^{sm}(R)}{E_{1n}^{sm}(R')}=\frac{R'}{R}\frac{h_n^{(1)}(\kappa R)}{h_n^{(1)}(\kappa R')}\frac{1+z_n^{(1)}(\kappa R)}{1+z_n^{(1)}(\kappa R')},\quad \frac{E_{2n}^{sm}(R)}{E_{2n}^{sm}(R')}=\frac{h_n^{(1)}(\kappa R)}{h_n^{(1)}(\kappa R')}.
\end{equation}
By \cite[$(2.39)$]{CK98}, for sufficiently large $n$, the spherical hankel functions admit the 
asymptotic behavior
\[
h_n^{(j)}(t)\sim(-1)^j {\rm i}\frac{(2n-1)!!}{t^{n+1}}.
\]
For any integer $n$, it is shown in \cite[Lemma 9.20]{Monk-2003} that
\begin{equation}\label{zn}
C_1 n\leq \big|1+z_n^{(1)}(t)\big|\leq C_2 n,
\end{equation}
where $C_1, C_2$ are constants independent of $n$. Substituting the above estimates into \eqref{ERR} completes the proof. 
\end{proof}

The following Birman--Solomyak decomposition theorem is useful in the subsequent analysis (cf. \cite{CWZ-2007-SIAM}).

\begin{lemma}\label{Decomposition}
For any $\boldsymbol{\phi}\in \boldsymbol{H}(\rm{curl}, \Omega)$, there exist a vector function
$\boldsymbol{\Phi}\in \boldsymbol{H}({\rm curl}, \Omega)\cap \boldsymbol{H}^1(\Omega)$
and a scalar function $\varphi\in H_0^1(\Omega)$ such that $\boldsymbol{\phi}=\boldsymbol{\Phi}+\nabla\varphi$, $\nabla\cdot\boldsymbol{\Phi}=0$, and
\[
\|\varphi\|_{H^1(\Omega)}+\|\boldsymbol{\Phi}\|_{\boldsymbol{H}^1(\Omega)}\leq C\|\boldsymbol{\phi}\|_{\boldsymbol{H}({\rm curl},\, \Omega)},
\] 
where $C>0$ is a constant depending on $\Omega$. 
\end{lemma}

Let $U_h$ be the standard $H^1$-conforming piecewise linear isoparametric finite element space over
$\mathcal{M}_h$. Lemma \ref{Key2} is concerned with the bounds for the first two terms of \eqref{Key}.
To prove it, we introduce the Cl\'ement interpolation operator
$\Pi_h: H_{\partial D}^1(\Omega)\rightarrow U_h$ and the Beck--Hiptmair--Hoppe--Wohlmuth interpolation 
operator $\mathscr{P}_h: \boldsymbol{H}^1(\Omega)\cap \boldsymbol{H}_{\partial D}({\rm curl}, \Omega)
\rightarrow \boldsymbol{V}_h$, which satisfy the following properties (cf. \cite{CC-mc08}): 
\begin{equation}\label{interpolation1}
	\|\phi-\Pi_h \phi\|_{L^2(K)}\leq C h_K\|\nabla\phi\|_{\boldsymbol{L}^2(\tilde{K})},\quad
	\|\phi-\Pi_h\phi\|_{L^2(F)}\leq Ch_F^{1/2}\|\nabla\phi\|_{\boldsymbol{L}^2(\tilde{F})},
\end{equation}
and
\begin{equation}\label{interpolation2}
	\|\boldsymbol{\Phi}-\mathscr{P}_h\boldsymbol{\Phi}\|_{\boldsymbol{L}^2(K)}\leq
	Ch_K\|\nabla\boldsymbol{\Phi}\|_{\boldsymbol{L}^2(\tilde{K})},\quad
	\|\boldsymbol{\Phi}-\mathscr{P}_h\boldsymbol{\Phi}\|_{\boldsymbol{L}^2(F)}\leq
	Ch_F^{1/2}\|\nabla\boldsymbol{\Phi}\|_{\boldsymbol{L}^2(\tilde{F})},
\end{equation}
where $\tilde{K}$ or $\tilde{F}$ is the union of elements on $\mathcal{M}_h$ with non-empty intersection
with $K$ or $F$, respectively.

\begin{lemma}\label{Key2}
For sufficiently large $N$ and for any $\boldsymbol{\psi}\in \boldsymbol{H}_{\partial D}({\rm curl}, \Omega)$, the following estimate holds: 
\begin{eqnarray*}
&&\left|a(\boldsymbol{\xi}, \boldsymbol{\psi})+{\rm i}\kappa\int_{\Gamma_R}
\left(\mathscr{T}-\mathscr{T}^N\right)\boldsymbol{\xi}_{\Gamma_R}
\cdot\overline{\boldsymbol{\psi}}_{\Gamma_R}{\rm d}s\right|\\
&&\lesssim\left(\Bigg(\sum\limits_{K\in\mathcal{M}_h}\eta_K^2\Bigg)^{1/2}
+\left(\frac{R'}{R}\right)^N\|\boldsymbol{f}\|_{TH^{-1/2}({\rm div}, \Gamma_R)}\right)
\|\boldsymbol{\psi}\|_{\boldsymbol{H}({\rm curl}, \Omega)}.
\end{eqnarray*}
\end{lemma}

\begin{proof}
A simple calculation shows that
	\begin{eqnarray*}
	&&a(\boldsymbol{\xi}, \boldsymbol{\psi})+{\rm i}\kappa\int_{\Gamma_R}
		(\mathscr{T}-\mathscr{T}^N)\boldsymbol{\xi}_{\Gamma_R}
		\cdot\overline{\boldsymbol{\psi}}_{\Gamma_R}{\rm d}s\\
	&=& a(\boldsymbol{E}, \boldsymbol{\psi})-a_h^N(\boldsymbol{E}_h^N, \boldsymbol{\psi}_h)
		+a_h^N(\boldsymbol{E}_h^N, \boldsymbol{\psi}_h)-a(\boldsymbol{E}_h^N, \boldsymbol{\psi})\\
	&& +{\rm i}\kappa\int_{\Gamma_R}(\mathscr{T}-\mathscr{T}^N)\boldsymbol{E}_{\Gamma_R}
		\cdot\overline{\boldsymbol{\psi}}_{\Gamma_R}{\rm d}s
		-{\rm i}\kappa\int_{\Gamma_R}(\mathscr{T}-\mathscr{T}^N)\left(\boldsymbol{E}^N_{h}\right)_{\Gamma_R}
		\cdot\overline{\boldsymbol{\psi}}_{\Gamma_R}{\rm d}s\\
	&=& (\boldsymbol{f}-\boldsymbol{f}^N, \boldsymbol{\psi})
		+(\boldsymbol{f}^N, \boldsymbol{\psi}-\boldsymbol{\psi}_h)
		-a_h^N(\boldsymbol{E}_h^N, \boldsymbol{\psi}-\boldsymbol{\psi}_h)
		+{\rm i}\kappa\int_{\Gamma_R}(\mathscr{T}-\mathscr{T}^N)\boldsymbol{E}_{\Gamma_R}
		\cdot\overline{\boldsymbol{\psi}}_{\Gamma_R}{\rm d}s	\\
	&=& (\boldsymbol{f}^N, \boldsymbol{\psi}-\boldsymbol{\psi}_h)
		-a_h^N(\boldsymbol{E}_h^N, \boldsymbol{\psi}-\boldsymbol{\psi}_h)
		+(\boldsymbol{f}-\boldsymbol{f}^N, \boldsymbol{\psi})
		+{\rm i}\kappa\int_{\Gamma_R}(\mathscr{T}-\mathscr{T}^N)\boldsymbol{E}_{\Gamma_R}
		\cdot\overline{\boldsymbol{\psi}}_{\Gamma_R}{\rm d}s\\
	&:=& J_1+J_2.
	\end{eqnarray*}
By \eqref{aNh}, we have 
	\begin{eqnarray*}
		J_1 &=& -a_h^N(\boldsymbol{E}_h^N, \boldsymbol{\psi}-\boldsymbol{\psi}_h)
			+(\boldsymbol{f}^N, \boldsymbol{\psi}-\boldsymbol{\psi}_h)\\
			&=& -\int_{\Omega} (\nabla\times\boldsymbol{E}_h^N )\cdot\left(\nabla\times
				(\overline{\boldsymbol{\psi}-\boldsymbol{\psi}_h})\right){\rm d}\boldsymbol{x}
				+\kappa^2\int_{\Omega}\boldsymbol{E}_h^N
					\cdot(\overline{\boldsymbol{\psi}-\boldsymbol{\psi}_h}){\rm d}\boldsymbol{x}\\
			&& +{\rm i}\kappa\int_{\Gamma_R} (\mathscr{T}^{N}\boldsymbol{E}_h^N)_{\Gamma_R}
				\cdot\overline{\left(\boldsymbol{\psi}-\boldsymbol{\psi}_h\right)_{\Gamma_R}}{\rm d}s
				+ (\boldsymbol{f}^N, \boldsymbol{\psi}-\boldsymbol{\psi}_h)\\
			&:=& J_1^1+J_1^2+J_1^3+J_1^4.
	\end{eqnarray*}
	
Denote the Birman--Solomyak decomposition of $\boldsymbol{\Psi}$ and $\boldsymbol{\Psi}_h$ by
	\[
		\boldsymbol{\psi}=\boldsymbol{\Phi}+\nabla\phi,\quad
		\boldsymbol{\psi}_h=\boldsymbol{\Phi}_h+\nabla\phi_h.
	\]
Using Green's identity and $\nabla\times\nabla\left(\phi-\phi_h\right)=0$, we get 
	\begin{eqnarray}\label{J11}
		J_1^1 &=& -\int_{\Omega} (\nabla\times\boldsymbol{E}_h^N) \cdot \left(\nabla\times
				(\overline{\boldsymbol{\psi}-\boldsymbol{\psi}_h})\right){\rm d}\boldsymbol{x}\notag\\
			&=& -\int_{\Omega} (\nabla\times\boldsymbol{E}_h^N )\cdot\left(\nabla\times
				(\overline{\boldsymbol{\Phi}-\boldsymbol{\Phi}_h})\right){\rm d}\boldsymbol{x}\notag\\	
			&=& -\sum\limits_{K\in\mathcal{M}_h}\int_{K} (\nabla\times\boldsymbol{E}_h^N) \cdot\left(\nabla\times
				(\overline{\boldsymbol{\Phi}-\boldsymbol{\Phi}_h})\right){\rm d}\boldsymbol{x}\notag\\
			&=& -\sum\limits_{K\in\mathcal{M}_h} \left[
				\int_{K}\nabla\times(\nabla\times\boldsymbol{E}_h^N )\cdot
				(\overline{\boldsymbol{\Phi}-\boldsymbol{\Phi}_h}){\rm d}\boldsymbol{x}
				+\int_{\partial K} (\nabla\times\boldsymbol{E}_h^N)\times\boldsymbol{\nu}
				\cdot (\overline{\boldsymbol{\Phi}-\boldsymbol{\Phi}_h}){\rm d}s
			\right].
	\end{eqnarray}
Similarly, we may deduce 
	\begin{eqnarray}\label{J12}
		J_1^2 &=& \kappa^2 \int_{\Omega} \boldsymbol{E}_h^N\cdot
			(\overline{\boldsymbol{\psi}-\boldsymbol{\psi}_h}){\rm d}\boldsymbol{x}\notag\\
			&=& \kappa^2 \int_{\Omega} \boldsymbol{E}_h^N\cdot
			(\overline{\boldsymbol{\Phi}-\boldsymbol{\Phi}_h}){\rm d}\boldsymbol{x}
			+\kappa^2 \int_{\Omega} \boldsymbol{E}_h^N\cdot
			(\overline{\nabla\phi-\nabla\phi_h}){\rm d}\boldsymbol{x}\notag\\
			&=&\sum\limits_{K\in\mathcal{M}_h} \left[ 
				\kappa^2\int_{K}  \boldsymbol{E}_h^N\cdot
				(\overline{\boldsymbol{\Phi}-\boldsymbol{\Phi}_h})
				-\nabla\cdot\boldsymbol{E}_h^N (\overline{\phi-\phi_h}){\rm d}\boldsymbol{x}
				+\kappa^2\int_{\partial K} \boldsymbol{E}_h^N\cdot\boldsymbol{\nu} 
				(\overline{\phi-\phi_h}){\rm d}s
			\right]
	\end{eqnarray}
and 
	\begin{eqnarray}\label{J13}
		J_1^3 &=& {\rm i}\kappa\int_{\Gamma_R} \mathscr{T}^{N}(\boldsymbol{E}_h^N)_{\Gamma_R}
				\cdot\left(\overline{\boldsymbol{\psi}}-\overline{\boldsymbol{\psi}_h}\right)_{\Gamma_R}{\rm d}s \notag\\
			&=& {\rm i}\kappa\int_{\Gamma_R} \mathscr{T}^{N}(\boldsymbol{E}_h^N)_{\Gamma_R}
				\cdot\left(\overline{\boldsymbol{\Phi}}-\overline{\boldsymbol{\Phi}_h}\right)_{\Gamma_R}{\rm d}s	
				+{\rm i}\kappa\int_{\Gamma_R} \mathscr{T}^{N}(\boldsymbol{E}_h^N)_{\Gamma_R}
				\cdot\left(\nabla\overline{\phi}-\nabla\overline{\phi_h}\right)_{\Gamma_R}{\rm d}s	\notag\\
			&=& \sum\limits_{K\in\mathcal{M}_h}\sum\limits_{F\subset \Gamma_R\cap \partial K}
			\bigg[{\rm i}\kappa \mathscr{T}^{N}(\boldsymbol{E}_h^N)_{\Gamma_R}
				\cdot(\overline{\boldsymbol{\Phi}}-\overline{\boldsymbol{\Phi}_h})_{\Gamma_R}\notag\\
			&& -{\rm i}\kappa\int_{\Gamma_R}
				{\rm div}_{\Gamma_R}\big(\mathscr{T}^{N}(\boldsymbol{E}_h^N)_{\Gamma_R}\big)
				(\overline{\phi}-\overline{\phi_h}){\rm d}s\bigg]. 
	\end{eqnarray}
A simple calculation shows that
	\begin{eqnarray}\label{J14}
	J_1^4 &=&  (\boldsymbol{f}^N, \boldsymbol{\psi}-\boldsymbol{\psi}_h) \notag\\
		    &=& \int_{\Gamma_R} \boldsymbol{f}^N
		    	\cdot(\overline{\boldsymbol{\Phi}-\boldsymbol{\Phi}_h}
		    	+\overline{\nabla\phi-\nabla\phi_h})_{\Gamma_R}{\rm d}s\notag\\
		    &=& 	\int_{\Gamma_R} \boldsymbol{f}^N
		    	\cdot\overline{(\boldsymbol{\Phi}-\boldsymbol{\Phi}_h)_{\Gamma_R}}{\rm d}s
		    	-\int_{\Gamma_R} {\rm div}_{\Gamma_R}\boldsymbol{f}^N
		    	(\overline{\phi-\phi_h}){\rm d}s.
	\end{eqnarray}
Substituting \eqref{J11}--\eqref{J14} into $J_1$ yields 
	\begin{eqnarray}\label{J1}
		J_1 &=& J_1^1+J_1^2+J_1^3+J_1^4 \notag\\
		 &=& \sum\limits_{K\in\mathcal{M}_h}\int_{K} \left[ 
		 	\kappa^2\boldsymbol{E}_h^N-\nabla\times(\nabla\times\boldsymbol{E}_h^N)
		 	\right]\cdot(\overline{\Phi}-\overline{\Phi}_h){\rm d}\boldsymbol{x}\notag\\
		 &&- \sum\limits_{K\in\mathcal{M}_h}\int_{K}
		 	\kappa^2 (\nabla\cdot\boldsymbol{E}_h^N) (\overline{\phi-\phi_h}){\rm d}\boldsymbol{x}
		 	- \sum\limits_{K\in\mathcal{M}_h}\int_{\partial K} (\nabla\times\boldsymbol{E}_h^N)\times\boldsymbol{\nu}
				\cdot (\overline{\boldsymbol{\Phi}-\boldsymbol{\Phi}_h}){\rm d}s \notag\\
		&& +	\sum\limits_{K\in\mathcal{M}_h}\kappa^2\int_{\partial K} \boldsymbol{E}_h^N\cdot\boldsymbol{\nu} 
				(\overline{\phi-\phi_h}){\rm d}s
				+ \sum\limits_{K\in\mathcal{M}_h}\sum\limits_{F\subset \Gamma_R\cap \partial K}
			\int_{F} {\rm i}\kappa \mathscr{T}^{N}(\boldsymbol{E}_h^N)_{\Gamma_R}
				\cdot(\overline{\boldsymbol{\Phi}}-\overline{\boldsymbol{\Phi}_h})_{\Gamma_R}{\rm d}s\notag\\
		&&-\sum\limits_{K\in\mathcal{M}_h}\sum\limits_{F\subset \Gamma_R\cap \partial K}
				\int_{F}{\rm i}\kappa{\rm div}_{\Gamma_R}\big(\mathscr{T}^{N}(\boldsymbol{E}^N_h)_{\Gamma_R}\big)
				(\overline{\phi}-\overline{\phi_h}){\rm d}s	\notag\\
		&&+	\sum\limits_{K\in\mathcal{M}_h}\sum\limits_{F\subset \Gamma_R\cap \partial K}
			\int_{F} \boldsymbol{f}^N
		    	\cdot\overline{(\boldsymbol{\Phi}-\boldsymbol{\Phi}_h)_{\Gamma_R}}{\rm d}s
		    	-\int_{F} {\rm div}_{\Gamma_R}\boldsymbol{f}^N
		    	(\overline{\phi-\phi_h}){\rm d}s \notag\\
		&=& \sum\limits_{K\in\mathcal{M}_h} \Bigg\{
				\int_{K} R_K^{(1)}\cdot (\overline{\boldsymbol{\Phi}}
					-\overline{\boldsymbol{\Phi}_h}){\rm d}\boldsymbol{x}
				+\int_{K} R_K^{(2)} (\overline{\phi}-\overline{\phi}_h){\rm d}\boldsymbol{x}\notag\\
		&&		+\sum\limits_{F\subset\partial K}\bigg[
					\int_{F}\frac{1}{2}J_F^{(1)}\cdot(\overline{\boldsymbol{\Phi}}
					-\overline{\boldsymbol{\Phi}_h})_{\Gamma_R}{\rm d}s
					+\int_{F}\frac{1}{2} J_F^{(2)}(\overline{\phi}-\overline{\phi}_h){\rm d}s
				\bigg]	
			\Bigg\}.	
	\end{eqnarray}
	
It follows from the Cl\'ement interpolation \eqref{interpolation1} and the Beck-Hiptmair-Hoppe-Wohlmuth interpolation \eqref{interpolation2} that we get 
\[
|J_1| \leq C\Bigg(\sum\limits_{K\in\mathcal{M}_h}\eta_K^2\Bigg)\|\boldsymbol{\psi}\|_{\boldsymbol{H}({\rm curl}, \Omega)}.
\]
On the other hand, we have from the definition of $\boldsymbol{f}$ and $\boldsymbol{f}^N$ that 
	\begin{eqnarray*}
		 \boldsymbol{f}-\boldsymbol{f}^N &=&
		 \left(\nabla\times\boldsymbol{E}^{\rm inc}\right)\times\boldsymbol{\nu}
		 	-{\rm i}\kappa\mathscr{T}\boldsymbol{E}_{\Gamma_R}^{\rm inc}
		 	-\left[\left(\nabla\times\boldsymbol{E}^{\rm inc}\right)\times\boldsymbol{\nu}
		 	-{\rm i}\kappa\mathscr{T}^N\boldsymbol{E}_{\Gamma_R}^{\rm inc}\right]\\
		 &=& -{\rm i}\kappa\left(\mathscr{T}-\mathscr{T}^N\right)\boldsymbol{E}^{\rm inc}_{\Gamma_R}.	
	\end{eqnarray*}
Using \eqref{CapacityOpe} and \eqref{TrunCapacity} yields
	\begin{eqnarray*}
		J_2 &=& (\boldsymbol{f}-\boldsymbol{f}^N, \boldsymbol{\psi})
		+{\rm i}\kappa\int_{\Gamma_R}(\mathscr{T}-\mathscr{T}^N)\boldsymbol{E}_{\Gamma_R}
		\cdot\overline{\boldsymbol{\psi}}_{\Gamma_R}{\rm d}s\\
		&=&{\rm i}\kappa\int_{\Gamma_R}(\mathscr{T}-\mathscr{T}^N)\boldsymbol{E}^s_{\Gamma_R}
		\cdot\overline{\boldsymbol{\psi}}_{\Gamma_R}{\rm d}s\\
			&=& {\rm i}\kappa\sum\limits_{n> N}\sum\limits_{|m|\leq n} \frac{{\rm i}\,\kappa R}{1+z_n^{(1)}(\kappa R)}
				E_{1n}^{sm}(R) \psi_{1n}^m(R)+\frac{1+z_n^{(1)}(\kappa R)}{{\rm i}\,\kappa R}E_{2n}^{sm}(R)\psi_{2n}^m(R)
				\\
			&=& {\rm i}\kappa\sum\limits_{n> N}\sum\limits_{|m|\leq n}
				 \frac{{\rm i}\kappa R}{1+z_n^{(1)}(\kappa R)}
				\frac{E_{1n}^{sm}(R)}{E_{1n}^{sm}(R')}E_{1n}^{sm}(R') \psi_{1n}^m(R)\notag\\
			&&	+\frac{1+z_n^{(1)}(\kappa R)}{{\rm i}\,\kappa R}
				\frac{E_{2n}^{sm}(R)}{E_{2n}^{sm}(R')}E_{2n}^{sm}(R')\psi_{2n}^m(R).
	\end{eqnarray*}
By Lemmas \ref{TraceRegularity} and \ref{ExponentialDecay}, and the asymptotic property \eqref{zn}, we have
	\begin{eqnarray*}
		|J_2| &\leq&\kappa\sum\limits_{n> N}\sum\limits_{|m|\leq n}
				\left| \frac{{\rm i}\,\kappa R}{1+z_n^{(1)}(\kappa R)}\right|
				\left|\frac{E_{1n}^m(R)}{E_{1n}^m(R')}\right|\left|E_{1n}^m(R')\right|\left| \psi_{1n}^m(R)\right|\\
&& +\left|\frac{1+z_n^{(1)}(\kappa R)}{{\rm i}\,\kappa R}\right|\left|\frac{E_{2n}^m(R)}{E_{2n}^m(R')}\right|
				\left|E_{2n}^m(R')\right|\left|\psi_{2n}^m(R)\right|\\
			&\lesssim& 	\kappa\sum\limits_{n> N}\sum\limits_{|m|\leq n}
				\frac{1}{n}\left(\frac{R'}{R}\right)^n\left|E_{1n}^m(R')\right|\left| \psi_{1n}^m(R)\right|
				+n\left(\frac{R'}{R}\right)^n\left|E_{2n}^m(R')\right|\left| \psi_{2n}^m(R)\right|\\
			&\lesssim& 	\left(\frac{R'}{R}\right)^N
				\left(\sum\limits_{n> N}\sum\limits_{|m|\leq n} 
				\frac{1}{\sqrt{1+n(n+1)}}|E_{1n}^m(R')|^2+\sqrt{1+n(n+1)}|E_{2n}^m(R')|^2\right)^{1/2}\\
			&& \times \left(\sum\limits_{n> N}\sum\limits_{|m|\leq n} 
				\frac{1}{\sqrt{1+n(n+1)}}|\psi_{1n}^m(R)|^2+\sqrt{1+n(n+1)}|\psi_{2n}^m(R)|^2\right)^{1/2}	\\
			&\lesssim &	\left(\frac{R'}{R}\right)^N \|\boldsymbol{E}\|_{TH^{-1/2}({\rm curl}, \Gamma_{R'})}
			\|\boldsymbol{\psi}\|_{TH^{-1/2}({\rm curl}, \Gamma_R)}
			\lesssim \left(\frac{R'}{R}\right)^N \|\boldsymbol{E}\|_{\boldsymbol{H}({\rm curl}, \Omega)}
			 \|\boldsymbol{\psi}\|_{\boldsymbol{H}({\rm curl}, \Omega)}.
	\end{eqnarray*}
The proof is completed by combining the estimate of $J_1$, $J_2$ and the inf-sup condition $\eqref{inf-sup}$.
\end{proof}

\begin{lemma}\label{imtbc}
For any $\boldsymbol{u}\in TH^{-1/2}({\rm curl}, \Gamma_R)$, the following estimate holds: 
\begin{equation*}
-\Im\langle \mathscr{F}^{N}\boldsymbol{u},\boldsymbol{u}\rangle_{\Gamma_R}\lesssim \delta\|\nabla\times\boldsymbol{u}\|_{\boldsymbol{L}^2(\Omega)}^2+C(\delta) \|\boldsymbol{u}\|^2_{\boldsymbol{L}^2(\Omega)},
\end{equation*}
where $\delta>0$ is a constant and $C(\delta)>0$ is also a constant depending on $\delta$. 
\end{lemma}

\begin{proof}
For any $\boldsymbol{u}\in TH^{-1/2}({\rm curl},\Omega)$, it has the Fourier expansion in $\Omega'$: 
\[
\boldsymbol{u}=\sum\limits_{n\in\mathbb{N}}\sum_{|m|\leq n}u_{1n}^m\boldsymbol{U}_n^m +\sum\limits_{n\in\mathbb{N}}\sum_{|m|\leq n}u_{2n}^m\boldsymbol{V}_n^m.
\] 
Taking the imaginary part of \eqref{TrunCapacity} gives
\[
\Im \langle \mathscr{T}^N\boldsymbol{u},\boldsymbol{u}\rangle_{\Gamma_R}
=\sum_{n=0}^{N}\sum_{|m|\leq n}\frac{\kappa R(1+\Re z_n^{(1)}(\kappa R))}{|1+z_n^{(1)}(\kappa R)|^2}
|u_{1n}^m|^2-\sum_{n=0}^{N}\sum_{|m|\leq n}\frac{1+\Re z_n^{(1)}(\kappa R)}{\kappa R}|u_{2n}^m|^2.
\]
It is shown in \cite[Lemma 9.2]{Monk-2003} that
\[
\Re z_n^{(1)}(\kappa R)\sim -n-1, \quad  \Re z_n^{(1)}(\kappa R)\leq -1.
\]
Hence, there exists a positive constant $C$ independent of $N$ such that
\[
\frac{1+\Re z_n^{(1)}(\kappa R)}{|1+z_n^{(1)}(\kappa R)|^2}| \geq -C\frac{1}{\sqrt{1+n(n+1)}}, 
\]
which leads to
\begin{eqnarray*}
\Im \langle \mathscr{T}^N\boldsymbol{u},\boldsymbol{u}\rangle_{\Gamma_R}
&\geq&\kappa R\sum_{n=0}^{N}\sum_{|m|\leq n}\frac{1+\Re z_n^{(1)}(\kappa R)}{|1+z_n^{(1)}(\kappa R)|^2}|u_{1n}^m|^2\\
&\geq&
-C\sum_{|m|\leq 1}\frac{1}{\sqrt{1+n(n+1)}}|u_{1n}^m|^2\geq -C\|\boldsymbol{u}\|^2_{TH^{-1/2}(\Gamma_R)}.
\end{eqnarray*}
The proof is completed by applying Lemma \ref{stability2}.
\end{proof}

To estimate the last term of \eqref{Key}, we introduce a dual problem.
Consider the Birman--Solomyak decomposition of $\boldsymbol{\xi}$: 
\[
\boldsymbol{\xi}=\nabla q+\boldsymbol{\zeta},\quad q\in H_0^1(\Omega),\quad \nabla\cdot\boldsymbol{\zeta}=0.
\]
The dual problem to \eqref{vform} is to find $\boldsymbol{W}\in \boldsymbol{H}_{\partial D}({\rm curl}, \Omega)$
such that it satisfies the variational problem
\begin{equation}\label{dualv}
a(\boldsymbol{\psi}, \boldsymbol{W})=(\boldsymbol{\psi}, \boldsymbol{\zeta})\quad \forall \boldsymbol{\psi}\in\boldsymbol{H}_{\partial D}({\rm curl}, \Omega).
\end{equation}
By the Helmholtz decomposition, it is easy to note that 
\begin{equation}\label{L2error}
	(\boldsymbol{\xi}, \boldsymbol{\xi})=(\boldsymbol{\xi}, \boldsymbol{\zeta})+(\boldsymbol{\xi}, \nabla q).
\end{equation}

\begin{lemma}\label{L22}
The following estimate holds: 
\begin{equation*}
\left|(\boldsymbol{\xi},\nabla q)\right|\lesssim\Bigg(\sum\limits_{K\in\mathcal{M}_n}\eta_K^2\Bigg)\|\nabla q\|_{\boldsymbol{L}^2(\Omega)}.
\end{equation*}
\end{lemma}

\begin{proof}
	For any $q_h\in U_h\cap H_0^1(\Omega)$, it is easy to check that $\nabla q_h\in \boldsymbol{V}_h$. Then we have
	\begin{eqnarray*}
		-\kappa^2\int_{\Omega}\boldsymbol{\xi}\cdot \nabla q_h {\rm d}\boldsymbol{x}
		&=& a(\boldsymbol{E}, \nabla q_h)-a_h^N(\boldsymbol{E}_h^N, \nabla q_h)\\
		&=& \int_{\Gamma_R} (\boldsymbol{f}-\boldsymbol{f}^N)\cdot(\overline{\nabla q_h})_{\Gamma_R}{\rm d}s \\
		&=& -\int_{\Gamma_R} {\rm div}_{\Gamma_R} (\boldsymbol{f}-\boldsymbol{f}^N) q_h{\rm d}s=0.
	\end{eqnarray*}
Taking $q_h=\Pi_h q$ and using the integration by parts, we get
	\begin{eqnarray*}
		\left|(\boldsymbol{\xi},\nabla q)\right| &=&\left|(\boldsymbol{\xi},\nabla( q-q_h))\right|\\
		&=& \left|\sum\limits_{K\in\mathcal{M}_h}\int_{K} \nabla\cdot \boldsymbol{E}_h^N (\overline{q-q_h}){\rm d}\boldsymbol{x}
		-\sum\limits_{K\in\mathcal{M}_h}\int_{\partial K} \boldsymbol{E}_h^N\cdot\boldsymbol{\nu} (\overline{q-q_h}){\rm d}s \right|\\
		&\lesssim& \left(
			\sum\limits_{K\in\mathcal{M}_n}\eta_K^2\right)\|\nabla q\|_{\boldsymbol{L}^2(\Omega)},
	\end{eqnarray*}
which completes the proof.	
\end{proof}

It remains to estimate the first term of \eqref{L2error}. It can be verified that the solution of dual problem \eqref{dualv}
satisfies the boundary value problem
\begin{equation}\label{W}
\begin{cases}
\nabla\times(\nabla\times \boldsymbol{W})-\kappa^2\boldsymbol{W}=\boldsymbol{\zeta}\quad &{\rm in}\,\Omega,\\
\boldsymbol{\nu}\times\boldsymbol{W}=0 & {\rm on}\,\partial D,\\
(\nabla\times\boldsymbol{W})\times\boldsymbol{\nu}=-{\rm i}\kappa\mathscr{T}^{*}\boldsymbol{W}_{\Gamma_R}
& {\rm on}\,\Gamma_R, 
\end{cases}
\end{equation}
where $\mathscr T^*$ is the adjoint operator to $\mathscr T$. Let the solution of \eqref{W} admit the Fourier expansion in $B_{R}\setminus \overline{B_{R'}}$: 
\begin{equation}\label{Wdual}
	\boldsymbol{W}=\sum\limits_{n\in\mathbb{N}}\sum\limits_{m=-n}^n
		w_{1n}^m(\rho) \boldsymbol{U}_n^m+w_{2n}^m(\rho) \boldsymbol{V}_n^m
		+w_{3n}^m(\rho) X_n^m \boldsymbol{e}_{\rho}.
\end{equation}

In Lemmas \ref{w2n} and \ref{w3n}, we present the ODE systems for the Fourier coefficients  $w_{2n}^m$ and $w_{3n}^m$. Lemma \ref{3Dsurface_odes} and \ref{3Dsurface_w1nw2nR} concern the solutions to the ODE systems. Finally, we deduce the estimate \eqref{L2error} in Lemma \ref{Lemma59}. 

\begin{lemma}\label{w2n}
If $\boldsymbol{W}$ satisfies \eqref{W}, then the coefficients $w_{2n}^{m}$ of \eqref{Wdual} satisfy the following ODE system:
\begin{equation*}
\begin{cases}
w_{2n}^{m''}(\rho)+\frac{2}{\rho}w_{2n}^{m'}(\rho)+\big(\kappa^2-\frac{n(n+1)}{\rho^2}\big)w_{2n}^m(\rho)
=-\zeta_{2n}^{m}(\rho),\quad &\rho\in(R', R),\\
w_{2n}^{m'}(R)-\frac{z_n^{(2)}(\kappa R)}{R} w_{2n}^m(R)=0, \quad &\rho=R,\\
w_{2n}^{m}(R')=w_{2n}^{m}(R'), \quad & \rho=R'. 
\end{cases}
\end{equation*}
\end{lemma}
 
\begin{proof}
It follows from \eqref{curlcurlUVX} that
\begin{align*}
\left[-\frac{1}{\rho}\frac{\partial^2}{\partial \rho^2}\left(\rho w_{2n}^m(\rho)\right)
+\frac{n(n+1)}{\rho^2} w_{2n}^m(\rho)-\kappa^2 w_{2n}^m(\rho)\right]V_n^m
=\zeta_{2n}^{m}(\rho) V_n^m,
\end{align*}
which gives 
\begin{equation}\label{3Dsurface_ode2}
w_{2n}^{m''}(\rho)+\frac{2}{\rho}w_{2n}^{m'}(\rho)+\Big(\kappa^2-\frac{n(n+1)}{\rho^2}\Big)w_{2n}^m(\rho)
=-\zeta_{2n}^{m}(\rho). 
\end{equation}
Substituting \eqref{3Dsurface_curlfe} into \eqref{CapacityOpe} leads to 
\begin{align*}
\frac{1}{\rho}\frac{\partial}{\partial\rho}\left(\rho w_{2n}^{m}(\rho)\right)\Bigg|_{\rho=R} V_n^m
=(-{\rm i}\kappa)\, {\rm i}\,\frac{1+z_n^{(2)}(\kappa R)}{\kappa R} w_{2n}^m(R) V_n^m,
\end{align*}
which shows 
\begin{equation}\label{3Dsurface_boundary2}
w_{2n}^{m'}(R)-\frac{z_n^{(2)}(\kappa R)}{R} w_{2n}^m(R)=0. 
\end{equation}
The proof is completed by combining \eqref{3Dsurface_ode2} and \eqref{3Dsurface_boundary2}. 
\end{proof} 

\begin{lemma}\label{w3n}
Let $v_{n}^m(\rho)=\rho w_{3n}^{m}(\rho)+C_n^m(R)$, where 
\[
C_n^m(R)=\frac{1+z_{n}^{(2)}(\kappa R)}{z_{n}^{(2)}(\kappa R)} \frac{R}{\kappa^2}\zeta_{3n}^m(R),
\] 
then $v_{n}^{m}$ satisfies the following ODE system:
\begin{equation*}
\begin{cases}
v_{n}^{m''}(\rho)+\frac{2}{\rho}v_{n}^{m'}(\rho)+\big(\kappa^2-\frac{n(n+1)}{\rho^2}\big)v_{n}^m(\rho)
=-\beta_{n}^{m}(\rho),\quad &\rho\in(R', R),\\
v_{n}^{m'}(R)-\frac{z_n^{(2)}(\kappa R)}{R} v_{n}^m(R)=0, \quad &\rho=R,\\
v_{n}^{m}(R')=v_{n}^{m}(R'), \quad &\rho=R',
\end{cases}
\end{equation*}
where
\[
\beta_n^{m}(\rho)=\rho\zeta_{3n}^m(\rho)-\Big(\kappa^2-\frac{n(n+1)}{\rho^2}\Big)C_n^m(R).
\]
\end{lemma}

\begin{proof}
	Letting $u_n^m(\rho)=\rho w_{3n}^{m}(\rho)$, we have from {\color{red}\eqref{ElectroFields}} and \eqref{curlcurlUVX} that 
	\begin{eqnarray}
	\left[-\frac{\sqrt{n(n+1)}}{\rho^2}\frac{\partial}{\partial\rho}\left(\rho w_{1n}^m(\rho)\right)
	+\frac{n(n+1)}{\rho^2} w_{3n}^m(\rho)-\kappa^2 w_{3n}^m(\rho)\right]X_n^m \boldsymbol e_{\rho}
	=\zeta_{3n}^{m}(\rho) X_n^m \boldsymbol e_{\rho}. \label{3Dsurface_ode3}
	\end{eqnarray}
Using Lemma \ref{3Dsurface_divfree} and eliminating $w_{1n}^m$ from \eqref{3Dsurface_ode3} yield 
\begin{equation}\label{diffu}
u_{n}^{m''}(\rho)+\frac{2}{\rho}u_n^{m'}(\rho)+\Big(\kappa^2-\frac{n(n+1)}{\rho^2}\Big)u_n^m(\rho)
= -\rho\zeta_{3n}^{m}(\rho).
\end{equation}
Noting $v_n^m(\rho)=u_n^m(\rho)+C_n^m(R)$, we get
\begin{equation*}\label{3Dsurface_v1}
v_{n}^{m''}(\rho)+\frac{2}{\rho}v_n^{m'}(\rho)+\Big(\kappa^2-\frac{n(n+1)}{\rho^2}\Big)v_n^m(\rho)= -\rho\zeta_{3n}^{m}(\rho)+\Big(\kappa^2-\frac{n(n+1)}{\rho^2}\Big) C_n^m(R).
\end{equation*}

It follows from \eqref{CapacityOpe} and \eqref{3Dsurface_curlfe} that
\begin{align*}
\left[\frac{1}{\rho}\frac{\partial}{\partial\rho}\left(\rho w_{1n}^m(\rho)\right)\Bigg|_{\rho=R}-\frac{\sqrt{n(n+1)}}{\rho}w_{3n}^m(\rho)\Bigg|_{\rho=R}\right]=(-{\rm i}\kappa)\frac{-{\rm i}\kappa R}{1+z_n^{(2)}(\kappa R)} w_{1n}^m(R),
\end{align*}
which gives
\begin{equation}
w_{1n}^{m'}(R)+\frac{1+z_n^{(2)}(\kappa R)+\kappa^2 R^2}{1+z_n^{(2)}(\kappa R)} \frac{1}{R}w_{1n}^{m}(R)
=\frac{\sqrt{n(n+1)}}{R} w_{3n}^{m}(R).	\label{3Dsurface_boundary1}
\end{equation}
Eliminating $w_{1n}^{m}(R)$ from \eqref{3Dsurface_boundary1} by using Lemma \ref{3Dsurface_divfree}, we obtain 
\begin{align*}
-\frac{1}{R}w_{1n}^{m}(R)&+\frac{1}{\sqrt{n(n+1)}}\frac{1}{R}\frac{\partial^2}{\partial \rho^2}\left(\rho u_n^m(\rho)\right)\Big|_{\rho=R}\\
& +\frac{1+z_n^{(2)}(\kappa R)+\kappa^2 R^2}{1+z_n^{(2)}(\kappa R)} \frac{1}{R}w_{1n}^{m}(R)
=\frac{\sqrt{n(n+1)}}{R} w_{3n}^{m}(R).
\end{align*}	
A simple calculation yields 
\begin{eqnarray*}
\frac{1}{\sqrt{n(n+1)}}\frac{\partial^2}{\partial \rho^2}\left(\rho u_n^m(\rho)\right)\Big|_{\rho=R}+\frac{\kappa^2 R}{1+z_n^{(2)}(\kappa R)}Rw_{1n}^{m}(R)=\sqrt{n(n+1)} w_{3n}^m(R),
\end{eqnarray*}
which gives
\begin{align*}
\frac{1}{\sqrt{n(n+1)}}\frac{\partial^2}{\partial \rho^2}\left(\rho u_n^m(\rho)\right)\Big|_{\rho=R}
+\frac{\kappa^2 R}{1+z_n^{(2)}(\kappa R)}\frac{1}{\sqrt{n(n+1)}}\frac{\partial}{\partial\rho}\left(\rho u_n^m(\rho)\right)\Big|_{\rho=R}\\
=\sqrt{n(n+1)} w_{3n}^m(R).
\end{align*}
Substituting \eqref{diffu} into the above equation leads to 
\begin{align*}
&\frac{1}{\sqrt{n(n+1)}}\left[R\left(\frac{n(n+1)}{R^2}-\kappa^2 \right)u_n^m(R)-R^2\zeta_{3n}^{m}(R)\right]\\
&\quad +\frac{\kappa^2 R}{1+z_n^{(2)}(\kappa R)}\frac{1}{\sqrt{n(n+1)}}\left[u_n^{m}(R)+Ru_n^{m'}(R)\right]
=\sqrt{n(n+1)}\frac{1}{R} u_{n}^m(R).
\end{align*}
It is easy to verify
\begin{eqnarray*}
\frac{\kappa^2 R^2}{1+z_n^{(2)}(\kappa R)} u_n^{m'}(R)+\left[\frac{\kappa^2 R}{1+z_n^{(2)}(\kappa R)}-\kappa^2 R\right] u_n^m(R)
=R^2 \zeta_{3n}^m(R), 
\end{eqnarray*}
which yields 
\begin{eqnarray*}
u_n^{m'}(R)-\frac{z_n^{(2)}(\kappa R)}{R} u_n^{m}(R)=\frac{1+z_n^{(2)}(\kappa R)}{\kappa^2}\zeta_{3n}^{m}(R).
\end{eqnarray*}
Noting $v_n^m(\rho)=u_n^m(\rho)+C_n^m(R)$ again, we get 
\begin{eqnarray*}
v_{n}^{m'}(R)-\frac{z_n^{(2)}(\kappa R)}{R} v_n^m(R)=0,
\end{eqnarray*}
which completes the proof. 
\end{proof}

Once $v_n^m$ is solved, $w_{1n}^m$ can be computed directly from Lemma \ref{3Dsurface_divfree} as follows: 
\begin{eqnarray}\label{3Dsurface_w1n}
	w_{1n}^{m}(\rho) &=& \frac{1}{\sqrt{n(n+1)}}\frac{1}{\rho}\frac{\partial}{\partial\rho}\left(\rho^2 w_{3n}^m(\rho)\right)\notag\\
	&=& \frac{1}{\sqrt{n(n+1)}}\frac{1}{\rho}\frac{\partial}{\partial\rho}\left(\rho\left(v_n^m(\rho)-C_n^m(R)\right)\right)\notag\\
	&=&  \frac{1}{\sqrt{n(n+1)}}\frac{1}{\rho} \left[v_n^m(\rho)+\rho v_{n}^{m'}(\rho)-C_n^m(R)\right].
\end{eqnarray}
Moreover, evaluating \eqref{3Dsurface_w1n} at $\rho=R$ yields 
\begin{eqnarray}\label{3Dsurface_w1nR}
	w_{1n}^{m}(R) &=& \frac{1}{\sqrt{n(n+1)}}\left[
		\frac{1}{R}v_n^m(R)+v_n^{m'}(R)-\frac{1}{R}C_n^m(R)
	\right]\notag\\
	&=&\frac{1}{\sqrt{n(n+1)}}\left[
		\frac{1}{R}v_n^m(R)+\frac{z_n^{(2)}(\kappa R)}{R} v_n^m(R)-\frac{1}{R}C_n^m(R)
	\right]\notag\\
	&=&\frac{1}{\sqrt{n(n+1)}}\frac{1}{R}\left[1+z_n^{(2)}(\kappa R)\right]v_n^m(R)
	-\frac{1}{\sqrt{n(n+1)}}\frac{1}{R} C_n^{m}(R).
\end{eqnarray}

The following results are concerned with the solutions to the ODE systems in Lemmas \ref{w2n} and \ref{w3n}. The proof can be found in \cite{BZHL-2020-DCDSS}. 

\begin{lemma}\label{3Dsurface_odes}
Let $v(\rho)$ satisfy the ODE system
\begin{equation*}
\begin{cases}
v''(\rho)+\frac{2}{\rho}v'(\rho)+\Big(\kappa^2-\frac{n(n+1)}{\rho^2}\Big)v(\rho)
=-\xi(\rho),\quad &\rho\in(R', R),\\
v'(R)-\frac{z_n^{(2)}(\kappa R)}{R} v(R)=0, \quad & \rho=R,\\
v(R')=v(R'),\quad &\rho=R',
\end{cases}
\end{equation*}
which has a unique solution given by 
\[
v(\rho)=S_n(\rho)v(R')+\frac{{\rm i}\kappa}{2}\int_{R'}^{\rho} t^2 W_n(\rho, t)\xi(t){\rm d}t
+\frac{{\rm i}\kappa}{2}\int_{R'}^{R} t^2 S_n(t)W_n(R', \rho)\xi(t){\rm d}t,
\]
where
\[
S_n(\rho)=\frac{h_n^{(2)}(\kappa\rho)}{h_n^{(2)}(\kappa R')},\quad
W_n(\rho, t)={\rm det}\begin{bmatrix}
h_n^{(1)}(\kappa\rho)	& h_n^{(2)}(\kappa\rho)\\
h_n^{(1)}(\kappa t)	& h_n^{(2)}(\kappa t)
\end{bmatrix}.
\]
Taking $\rho=R$ in the solution yields 
\[
v(R)=S_n(R) v(R')+\frac{{\rm i}\kappa}{2}\int_{R'}^{R} t^2 S_n(R)W_n(R', t)\xi(t){\rm d}t.
\]
Moreover, it follows from the asymptotic property of $h_n^{(2)}(t)$ that
\[
|S_n(R)|\lesssim \left(\frac{R'}{R}\right)^n, \quad
|W_n(R', t)|\lesssim \frac{1}{n} \left(\frac{t}{R'}\right)^n,
\quad n\rightarrow\infty.
\]
\end{lemma}

The estimates of $w_{1n}^{m}$ and $w_{2n}^m$ at $\rho=R$ are given in the following lemma. 

\begin{lemma}\label{3Dsurface_w1nw2nR}
Let  $w_{1n}^{m}$ and $w_{2n}^m$ be the Fourier coefficients of $\boldsymbol{W}$. They satisfy the estimates
\begin{equation*}
|w_{2n}^{m}(R)|\lesssim \left(\frac{R'}{R}\right)^n |w_{2n}^{m}(R')|+\frac{1}{n^2}\|\zeta_{2n}^m(t)\|_{L^{\infty}([R', R])}
\end{equation*}
and
\begin{equation*}
|w_{1n}^m(R)| \lesssim \left(\frac{R'}{R}\right)^n |w_{3n}^m(R')|+\frac{1}{n^2}\|\zeta_{3n}^m\|_{L^{\infty}([R', R])}
+ |\zeta_{3n}^m(R) |.
\end{equation*}	
\end{lemma}

\begin{proof}
It follows from Lemmas \ref{w2n} and \ref{3Dsurface_odes} that
\begin{equation*}
	|w_{2n}^{m}(R)|\lesssim \left(\frac{R'}{R}\right)^n |w_{2n}^{m}(R')|+\frac{1}{n^2}\|\zeta_{2n}^m(t)\|_{L^{\infty}([R', R])}.
\end{equation*}
In addition, we have from Lemma \ref{3Dsurface_w1nw2nR} that
\begin{equation*}
v_n^m(R)=S_n(R) v_n^m(R')+\frac{{\rm i}\kappa}{2}\int_{R'}^{R} t^2 S_n(R)W_n(R', t)\beta_n^m(t){\rm d}t.
\end{equation*}
Substituting the above equation into \eqref{3Dsurface_w1nR}, we obtain 
\begin{eqnarray*}
	w_{1n}^{m}(R)  &=& \frac{1}{\sqrt{n(n+1)}}\frac{1}{R}\left[1+z_n^{(2)}(\kappa R)\right]v_n^m(R)
		-\frac{1}{\sqrt{n(n+1)}}\frac{1}{R} C_n^{m}(R) \\
	&= & \frac{1}{\sqrt{n(n+1)}}\frac{1}{R}\left[1+z_n^{(2)}(\kappa R)\right]S_n(R) v_n^m(R')
	-\frac{1}{\sqrt{n(n+1)}}\frac{1}{R} C_n^{m}(R)\\
	&& +\frac{1}{\sqrt{n(n+1)}}\frac{1}{R}\left[1+z_n^{(2)}(\kappa R)\right]
			\frac{{\rm i}\,\kappa}{2}\int_{R'}^{R} t^2 S_n(R)W_n(R', t)\beta_{n}^m(t)\,{\rm d}t\\
	&=& \frac{1}{\sqrt{n(n+1)}}\frac{1}{R}\left\{
		\left[1+z_n^{(2)}(\kappa R)\right]S_n(R)\left[R' w_{3n}^m(R')+C_n^m(R)\right]-C_n^m(R)
	\right\}\\
	&& +\frac{{\rm i}\kappa}{2}\frac{1}{\sqrt{n(n+1)}}\frac{1}{R}\left[1+z_n^{(2)}(\kappa R)\right]
			\int_{R'}^{R} t^3 S_n(R)W_n(R', t)\zeta_{3n}^m(t){\rm d}t\\
	&& -\frac{{\rm i}\kappa}{2}\frac{1}{\sqrt{n(n+1)}}\frac{1}{R}\left[1+z_n^{(2)}(\kappa R)\right]	
		\int_{R'}^{R} t^2 S_n(R)W_n(R', t)\Big(\kappa^2-\frac{n(n+1)}{t^2}\Big)C_n^m(R){\rm d}t,
\end{eqnarray*}
which can be simplified to 
\begin{eqnarray*}
	w_{1n}^{m}(R)  &=& \frac{1}{\sqrt{n(n+1)}}\frac{R'}{R}
		\left[1+z_n^{(2)}(\kappa R)\right]S_n(R) w_{3n}^m(R') \\
		&& +\frac{{\rm i}\kappa}{2}\frac{1}{\sqrt{n(n+1)}}\frac{1}{R}\left[1+z_n^{(2)}(\kappa R)\right] S_n(R)
			\int_{R'}^{R} t^3W_n(R', t)\zeta_{3n}^m(t){\rm d}t\\
		&& -\frac{1}{\sqrt{n(n+1)}}\frac{1}{R} C_n^m(R)
		+\frac{1}{\sqrt{n(n+1)}}\frac{1}{R}\left[1+z_n^{(2)}(\kappa R)\right]S_n(R) C_n^m(R)\\
		&&\times\bigg\{
		1-\frac{{\rm i}\kappa^3}{2}\int_{R'}^{R} t^2 W_n(R', t){\rm d}t+\frac{{\rm i}\kappa}{2}n(n+1)\int_{R'}^{R}  W_n(R', t){\rm d}t\bigg\}. 
\end{eqnarray*}

Using the asymptotic expansion (cf. \cite[pp. 12]{BZHL-2020-DCDSS})
\[
W_n(R', t)\sim -\frac{2{\rm i}}{(2n+1) \kappa R'}\left(\frac{t}{R'}\right)^n,\quad n\rightarrow\infty,
\]
we get from straightforward calculations that 
\begin{eqnarray*}
\int_{R'}^{R}  W_n(R', t){\rm d}t &\sim&  -\frac{2{\rm i}}{(2n+1)(n+1)}\frac{1}{\kappa}\left(\frac{R}{R'}\right)^{n+1},\\
\int_{R'}^{R} t^2 W_n(R', t){\rm d}t &\sim&  -\frac{2{\rm i}}{(2n+1)(n+3)}\frac{R^2}{\kappa}\left(\frac{R}{R'}\right)^{n+1}.
\end{eqnarray*}
Substituting the above equations into $w_{1n}^m(R)$ yields 
\begin{eqnarray*}
		&&1-\frac{{\rm i}\kappa^3}{2}\int_{R'}^{R} t^2 W_n(R', t){\rm d}t
		+\frac{{\rm i}\,\kappa}{2}(n+1)n\int_{R'}^{R}  W_n(R', t){\rm d}t\\
		&&\sim 1-\frac{\kappa^2 R^2}{(2n+1)(n+3)}\left(\frac{R}{R'}\right)^{n+1}
		+\frac{n}{2n+1}\left(\frac{R}{R'}\right)^{n+1}, 
\end{eqnarray*} 
which gives 
\begin{eqnarray*}
	&&\Bigg|\frac{1}{\sqrt{n(n+1)}}\frac{1}{R}	\left[1+z_n^{(2)}(\kappa R)\right]S_n(R) C_n^m(R)\bigg\{
		1-\frac{{\rm i}\,\kappa^3}{2}\int_{R'}^{R} t^2 W_n(R', t){\rm d}t \\
		&& \quad +\frac{{\rm i}\kappa}{2}\sqrt{(n+1)n}\int_{R'}^{R}  W_n(R', t){\rm d}t
		\bigg\}-\frac{1}{\sqrt{n(n+1)}}\frac{1}{R} C_n^m(R)\Bigg|\\
	&&\lesssim 	\left|\frac{1+z_n^{(2)}(\kappa R)}{2n+1}\frac{1}{R'}C_n^m(R)\right|.
\end{eqnarray*}

It is shown in \cite[Lemma 3.1]{LY-ipi} that
\[
z_n(t)=-(n+1)+\frac{t^4}{16n}+\frac{t^2}{2n}+O\left(\frac{1}{n^2}\right).
\]
Hence 
\[
	\Bigg|\frac{1+z_n^{(2)}(\kappa R)}{2n+1}\frac{1}{R'}C_n^m(R)
	\Bigg|\lesssim \frac{1}{2 R'}\left|C_n^m(R)\right|.
\]
Plugging the above equation into to $w_{1n}^m(R)$, we obtain 
\begin{eqnarray*}
	|w_{1n}^m(R)| \lesssim \left(\frac{R'}{R}\right)^n |w_{3n}^m(R')|
	+\frac{1}{n^2}\|\zeta_{3n}^m\|_{L^{\infty}([R', R])}+\left|C_n^m(R)\right|.
\end{eqnarray*}
Noting
\[
\left|C_n^m(R)\right|=\left|\frac{1+z_{n}^{(2)}(\kappa R)}{z_{n}^{(2)}(\kappa R)} \frac{R}{\kappa^2}\zeta_{3n}^m(R)\right|
\sim \frac{R}{\kappa^2} |\zeta_{3n}^m(R)|,
\] 
we have 
\begin{eqnarray*}
	|w_{1n}^m(R)| 
	\lesssim \left(\frac{R'}{R}\right)^n |w_{3n}^m(R')|+\frac{1}{n^2}\|\zeta_{3n}^m\|_{L^{\infty}([R', R])}
	+|\zeta_{3n}^m(R)|.
\end{eqnarray*}
The estimate for $w_{2n}^{m}(R)$ can be obtained by following the same steps.
\end{proof}

The following result is crucial to prove Theorem \ref{thm}.

\begin{lemma}\label{Lemma59}
Let $\boldsymbol{W}$ be the solution of the dual problem. Then the following estimate holds: 
	\[
		\kappa\left|\int_{\Gamma_R} \left(\mathscr{T}-\mathscr{T}^N\right)\boldsymbol{\xi}_{\Gamma_R}\cdot
		\overline{\boldsymbol{W}}_{\Gamma_R}\right|\lesssim
		\frac{1}{N}\|\boldsymbol{\xi}\|_{H({\rm curl}\,\Omega)}^2.
	\]
\end{lemma}

\begin{proof}
Using \eqref{CapacityOpe} and \eqref{TrunCapacity}, we have 
\begin{eqnarray*}
	\kappa\left|\int_{\Gamma_R} \left(\mathscr{T}-\mathscr{T}^N\right)\boldsymbol{\xi}_{\Gamma_R}\cdot
	\overline{\boldsymbol{W}}_{\Gamma_R}\right| &\leq &\kappa\left|
	\sum\limits_{n=N+1}\sum\limits_{|m|\leq n} \frac{{\rm i}\,\kappa R}{1+z_n^{(1)}(\kappa R)}
	\xi_{1n}^m(R)\overline{w}_{1n}^m(R)\right| \\
	&& +\kappa\left|\sum\limits_{n=N+1}\sum\limits_{|m|\leq n} \frac{1+z_n^{(1)}(\kappa R)}{{\rm i}\,\kappa R}
	\xi_{2n}^m(R)\overline{w}_{2n}^m(R)\right|.
\end{eqnarray*}
It follows from the Cauchy--Schwarz inequality that
\begin{eqnarray*}
	&&\sum\limits_{n=N+1}\sum\limits_{|m|\leq n}
	\left|\frac{1+z_n^{(1)}(\kappa R)}{\kappa R}\right| \left|\xi_{2n}^m(R)\right| 
	\left|w_{2n}^m(R)\right|\\
	&&\leq \sum\limits_{n=N+1}\sum\limits_{|m|\leq n}
	\left|\frac{1+z_n^{(1)}(\kappa R)}{\kappa R}\right|\left(1+n(n+1)\right)^{1/4} 
	\left|\xi_{2n}^m(R)\right| \left(1+n(n+1)\right)^{-1/4} \left|w_{2n}^m(R)\right|\\
	&&\leq \frac{1}{N^2} \left[\sum\limits_{n=N+1}\sum\limits_{|m|\leq n}
		\sqrt{1+n(n+1)}\left|\xi_{2n}^m(R)\right|^2 \right]^{1/2}\\
	&&\quad\times\left[\sum\limits_{n=N+1}\sum\limits_{|m|\leq n}
		\left|\frac{1+z_n^{(1)}(\kappa R)}{\kappa R}\right|^2
		\frac{N^4}{\sqrt{1+n(n+1)}}\left|w_{2n}^m(R)\right|^2 \right]^{1/2}\\
	&&\leq \frac{1}{N^2} \|\boldsymbol{\xi}\|_{TH^{-\frac{1}{2}}({\rm curl}, \Gamma_R)}
	\left[\sum\limits_{n=N+1}\sum\limits_{|m|\leq n} n^5\left|w_{2n}^m(R)\right|^2\right]^{1/2}.
\end{eqnarray*}
By Lemma \ref{3Dsurface_w1nw2nR}, we have 
\begin{eqnarray}\label{dualn51}
	\sum\limits_{n=N+1}\sum\limits_{|m|\leq n} n^5\left|w_{2n}^m(R)\right|^2
	&=&\sum\limits_{n=N+1}\sum\limits_{|m|\leq n} n^5\left[ 
	\left(\frac{R'}{R}\right)^{2n}|w_{2n}^m(R')|^2+\frac{1}{n^4}\|\zeta_{2n}^m(t)\|^2_{L^{\infty}([R', R])}\right]\notag\\
	&=&\sum\limits_{n=N+1}\sum\limits_{|m|\leq n} 
	n^5\left(\frac{R'}{R}\right)^{2n}|w_{2n}^m(R')|^2
	+n\|\zeta_{2n}^m(t)\|^2_{L^{\infty}([R', R])}. 
\end{eqnarray}
We have from Lemma \ref{TraceRegularity} that
\begin{eqnarray}\label{dualn52}
	&&\sum\limits_{n=N+1}\sum\limits_{|m|\leq n} 
	n^5\left(\frac{R'}{R}\right)^{2n}|w_{2n}^m(R')|^2 \lesssim
	\max\left(n^4\left(\frac{R'}{R}\right)^{2n}\right)\|\boldsymbol{\boldsymbol{W}}\|^2_{TH^{-1/2}({\rm curl}, \Gamma_R)}\notag\\
	&&\lesssim \max\left(n^4\left(\frac{R'}{R}\right)^{2n}\right)
		\|\boldsymbol{\boldsymbol{W}}\|^2_{\boldsymbol{H}({\rm curl},\, \Omega)}
	\lesssim \max\left(n^4\left(\frac{R'}{R}\right)^{2n}\right)
		\|\boldsymbol{\boldsymbol{\zeta}}\|^2_{\boldsymbol{H}^1(\Omega)}.
\end{eqnarray}
It is shown in \cite{JLLZ-JSC-2017} that
\begin{equation}\label{3Dsurface_Linfty}
	\|\zeta(t)\|^2_{L^{\infty}([R', R])}
	\leq \left(\frac{2}{R-R'}+n\right)\|\zeta(t)\|^2_{L^2([R', R])}+\frac{1}{n}\|\zeta'(t)\|^2_{L^2([R', R])}.
\end{equation}
Moreover,
\begin{eqnarray*}
\|\zeta_{2n}^{m'}(t)\|^2_{L^2(R', R)} &\leq& 
\left(\frac{1}{R'}\right)^2 \|t\zeta_{2n}^{m'}(t)\|_{L^2(R', R)}^2\\
&\leq& \left(\frac{1}{R'}\right)^2 \|t\zeta_{2n}^{m'}(t)+\zeta_{2n}^m(t)\|^2_{L^2(R', R)}
+\left(\frac{1}{R'}\right)^2\|\zeta_{2n}^m(t)\|^2_{L^2(R', R)}.
\end{eqnarray*}
Combining \eqref{semicurl} and \eqref{3Dsurface_Linfty} leads to 
\begin{eqnarray}\label{dualn53}
	&& \sum\limits_{n=N+1}\sum\limits_{|m|\leq n}n\|\zeta_{2n}^m(t)\|^2_{L^{\infty}([R', R])} \notag\\
	&&\leq \sum\limits_{n=N+1}\sum\limits_{|m|\leq n}\left[\left(\frac{2}{R-R'}+n\right)n
	+\left(\frac{1}{R'}\right)^2\right]\|\zeta_{2n}^m(t)\|^2_{L^2([R', R])}\notag\\
	&&\quad +\left(\frac{1}{R'}\right)^2\|t\zeta_{2n}^{m'}(t)+\zeta_{2n}^m(t)\|^2_{L^2([R', R])}\notag\\
	&&\lesssim \|\nabla\times\boldsymbol{\zeta}\|_{L^2(\Omega)}^2.
\end{eqnarray}
Since $\max\left(n^4\left(\frac{R'}{R}\right)^{2n}\right)$ is bounded, we have
from Lemma \ref{TraceRegularity}, \eqref{dualn51}--\eqref{dualn52} and \eqref{dualn53} that
\begin{eqnarray}\label{3Dsurface_lemma59w2n}
	\sum\limits_{n=N+1}\sum\limits_{|m|\leq n}
	\left|\frac{1+z_n^{(1)}(\kappa R)}{\kappa R}\right| \left|\xi_{2n}^m(R)\right| 
	\left|w_{2n}^m(R)\right|
	 \lesssim  \frac{1}{N^2} \|\boldsymbol{\xi}\|^2_{H({\rm curl}, \Omega)}. 
\end{eqnarray}

Next is to estimate $\xi_{1n}^m(R)\overline{w}_{1n}^m(R)$. It follows from Lemma 
\ref{3Dsurface_w1nw2nR} that
\begin{eqnarray*}
	&&\left|\sum\limits_{n=N+1}\sum\limits_{|m|\leq n} \frac{{\rm i}\,\kappa R}{1+z_n^{(1)}(\kappa R)}
	 \xi_{1n}^m(R)\overline{w}_{1n}^m(R)\right|\\
	 && \lesssim \sum\limits_{n=N+1}\sum\limits_{|m|\leq n} \frac{1}{n} \left|\xi_{1n}^{m}(R)\right|
	 \left|w_{1n}^m(R)\right|\\
	 && \lesssim \sum\limits_{n=N+1}\sum\limits_{|m|\leq n} \frac{1}{n} \left|\xi_{1n}^{m}(R)\right|
	 \left[ \left(\frac{R'}{R}\right)^n |w_{3n}^m(R')|+\frac{1}{n^2}\|\zeta_{3n}^m\|_{L^{\infty}([R', R])}
	+ |\zeta_{3n}^m(R) |\right]\\
	&&:= I_1+I_2+I_3.
\end{eqnarray*}
A straightforward calculation yields
\begin{eqnarray*}
	I_1 &=& \sum\limits_{n=N+1}\sum\limits_{|m|\leq n} \frac{1}{n}\left(\frac{R'}{R}\right)^n
	\left|\xi_{1n}^{m}(R)\right||w_{3n}^m(R')|\\
	&=& \sum\limits_{n=N+1}\sum\limits_{|m|\leq n}\frac{1}{n^2} \frac{1}{\sqrt{n}}
	\left|\xi_{1n}^{m}(R)\right|
	\left(\frac{R'}{R}\right)^n n^{\frac{3}{2}}|w_{3n}^m(R')|\\
	&\lesssim& \frac{1}{N^2} \left( \sum\limits_{n=N+1}\sum\limits_{|m|\leq n}
	\frac{1}{\sqrt{1+n(n+1)}}\left|\xi_{1n}^{m}(R)\right|^2
	\right)^{1/2}  \left( \sum\limits_{n=N+1}\sum\limits_{|m|\leq n}
	\left(\frac{R'}{R}\right)^{2n} n^{3}|w_{3n}^m(R')|^2\right)^{1/2} \\
	&\lesssim& \frac{1}{N^2} \|\boldsymbol{\xi}\|_{TH^{-1/2}({\rm curl}, \Gamma_R)}
	\max\left(n^2\left(\frac{R'}{R}\right)^n\right) \left( \sum\limits_{n=N+1}\sum\limits_{|m|\leq n}
	\frac{1}{\sqrt{n(n+1)}}|w_{3n}^m(R')|^2\right)^{1/2} \\
	&\lesssim& \frac{1}{N^2} \|\boldsymbol{\xi}\|_{\boldsymbol{H}({\rm curl}, \Omega)}
	\|\boldsymbol{W}\|_{\boldsymbol{H}^{-1/2}(\Gamma_{R'})}\lesssim  \frac{1}{N^2}\|\boldsymbol{\xi}\|^2_{H({\rm curl}, \Omega)}.
\end{eqnarray*}
Since $\nabla\cdot\boldsymbol{\zeta}=0$, it can be obtained from Lemma \ref{3Dsurface_divfree} that
\begin{eqnarray*}
	\zeta_{3n}^{m'}(\rho)+\frac{2}{\rho}\zeta_{3n}^m(\rho)=\sqrt{n(n+1)}\frac{1}{\rho}\zeta_{1n}^m(\rho).
\end{eqnarray*}
Then we have
\begin{eqnarray*}
	\int_{R'}^R \left|\zeta_{3n}^{m'}(t)\right|^2\,{\rm d}t &=&
	\int_{R'}^R \left|\sqrt{n(n+1)}\frac{1}{\rho}\zeta_{1n}^m(\rho)
	-\frac{2}{\rho}\zeta_{3n}^m(\rho)\right|^2\,{\rm d}t\\
	&\lesssim& n(n+1) \|\zeta_{1n}^m\|_{L^2([R', R])}^2+\|\zeta_{3n}^m\|_{L^2([R', R])}^2.
\end{eqnarray*}
Substituting the above equation into \eqref{3Dsurface_Linfty} gives 
\begin{eqnarray*}
	\|\zeta_{3n}^m(t)\|^2_{L^{\infty}([R', R])}
	&\leq& \left(\frac{2}{R-R'}+n\right)\|\zeta_{3n}^m(t)\|^2_{L^2([R', R])}+\frac{1}{n}\|\zeta_{3n}^{m'}(t)\|^2_{L^2([R', R])}\\
	&\lesssim& \left(\frac{2}{R-R'}+n\right)\|\zeta_{3n}^m(t)\|^2_{L^2([R', R])}
	+n\|\zeta_{1n}^m\|_{L^2([R', R])}^2.
\end{eqnarray*}
For $I_2$, we have 
\begin{eqnarray*}
	I_2 &=& \sum\limits_{n=N+1}\sum\limits_{|m|\leq n}\frac{1}{n^3}
	 \left|\xi_{1n}^{m}(R)\right|\|\zeta_{3n}^m\|_{L^{\infty}([R', R])}\\
	 &\lesssim& \frac{1}{N^2}  \left( \sum\limits_{n=N+1}\sum\limits_{|m|\leq n}
	\frac{1}{\sqrt{n(n+1)}}\left|\xi_{1n}^{m}(R)\right|^2
	\right)^{1/2} \left( \sum\limits_{n=N+1}\sum\limits_{|m|\leq n}
	\frac{1}{n}\|\zeta_{3n}^m\|_{L^{\infty}([R', R])}^2 \right)^{1/2}\\
	&\lesssim& \frac{1}{N^2} \|\boldsymbol{\xi}\|_{\boldsymbol{H}({\rm curl},\,\Omega)} 
	\|\boldsymbol{\zeta}\|_{L^2(\Omega)}
	\leq  \frac{1}{N^2} \|\boldsymbol{\xi}\|_{H(\rm curl,\,\Omega)}^2.
\end{eqnarray*}
Similarly, it can be shown that 
\begin{eqnarray*}
	I_3 &=&  \sum\limits_{n=N+1}\sum\limits_{|m|\leq n}\frac{1}{n}
	 \left|\xi_{1n}^{m}(R)\right| \left|\zeta_{3n}^{m}(R)\right|\\
	 &=&  \sum\limits_{n=N+1}\sum\limits_{|m|\leq n} \frac{1}{n}
	 (1+n(n+1))^{-1/4} \left|\xi_{1n}^{m}(R)\right| (1+n(n+1))^{1/4} \left|\zeta_{3n}^{m}(R)\right|\\
	 &\lesssim& \frac{1}{N}\left( \sum\limits_{n=N+1}\sum\limits_{|m|\leq n}
	\frac{1}{\sqrt{1+n(n+1)}}\left|\xi_{1n}^{m}(R)\right|^2
	\right)^{1/2} \left( \sum\limits_{n=N+1}\sum\limits_{|m|\leq n}
	\sqrt{1+n(n+1)}\left|\zeta_{3n}^{m}(R)\right|^2
	\right)^{1/2} \\
	&\lesssim& \frac{1}{N} \|\boldsymbol{\xi}\|_{H^{-1/2}(\Gamma_R)} 
	\|\boldsymbol{\zeta}\|_{H^{1/2}(\Gamma_R)}
	\lesssim  \frac{1}{N} \|\boldsymbol{\xi}\|^2_{H({\rm curl}, \Omega)}.
\end{eqnarray*}

Combining the estimates of $I_1$, $I_2,$ $I_3$ and Lemma \ref{Decomposition}, we obtain 
\begin{equation}\label{3Dsurface_Lemma59w1}
	\left|\sum\limits_{n=N+1}\sum\limits_{|m|\leq n} \frac{{\rm i}\,\kappa R}{1+z_n^{(1)}(\kappa R)}
	 \xi_{1n}^m(R)\overline{w}_{1n}^m(R)\right|
	 \lesssim \frac{1}{N}  \|\boldsymbol{\xi}\|_{\boldsymbol{H}({\rm curl}, \Omega)}^2.
\end{equation}
It follows from \eqref{3Dsurface_lemma59w2n} and \eqref{3Dsurface_Lemma59w1} that
\begin{eqnarray*}
\kappa\left|\int_{\Gamma_R} \left(\mathscr{T}-\mathscr{T}^N\right)\boldsymbol{\xi}_{\Gamma_R}\cdot
\overline{\boldsymbol{W}}_{\Gamma_R}\right|\leq\frac{1}{N} \|\boldsymbol{\xi}\|^2_{H({\rm curl}, \Omega)},
\end{eqnarray*}
which completes the proof.
\end{proof}

Combining \eqref{dualv}--\eqref{L2error}, \eqref{L22}, and Lemmas {\color{red}\ref{L22}} and \ref{Lemma59}, we obtain 
\begin{eqnarray}\label{L2bounds}
	\|\boldsymbol{\xi}\|_{\boldsymbol{L}^2(\Omega)} &=& \left|\left(\boldsymbol{\xi}, \boldsymbol{\zeta}\right)
	+(\boldsymbol{\xi}, \nabla q)\right|\notag\\
	&=& \left|a(\boldsymbol{\xi}, \boldsymbol{W})
		+{\rm i}\kappa\int_{\Gamma_R}\left(\mathscr{T}-\mathscr{T}^N\right)\boldsymbol{\xi}_{\Gamma_R}
		\cdot\overline{\boldsymbol{W}_{\Gamma_R}}{\rm d}s\right|\notag\\
	&&	+\left|{\rm i}\kappa\int_{\Gamma_R}\left(\mathscr{T}-\mathscr{T}^N\right)\boldsymbol{\xi}_{\Gamma_R}
		\cdot\overline{\boldsymbol{W}_{\Gamma_R}}{\rm d}s\right|+\left|(\boldsymbol{\xi}, \nabla q)\right|\notag\\
	&\lesssim& \frac{1}{N}\|\boldsymbol{\xi}\|^2_{\boldsymbol{H}({\rm curl}, \Omega)}
	+	\Bigg(\Bigg(\sum\limits_{K\in\mathcal{M}_h}\eta_K^2\Bigg)^{1/2}
		+\left(\frac{R'}{R}\right)^N\|\boldsymbol{f}\|_{TH^{-1/2}({\rm div}, \Gamma_R)}\Bigg)\notag\\
	&&\times	\|\boldsymbol{\xi}\|_{\boldsymbol{H}({\rm curl}, \Omega)}.
\end{eqnarray}

Now we are ready to prove Theorem \ref{thm}. 

\begin{proof}
It follows from the error representation formula \eqref{Key} and Lemmas \ref{Key2} and \ref{imtbc} that
\begin{eqnarray*}
\|\boldsymbol{\xi}\|_{H({\rm curl},\Omega)}^2
&\leq& C\Bigg(\Bigg(\sum_{T\in\mathcal{M}_h}\eta_T^2\Bigg)^{1/2}
+\left(\frac{R'}{R}\right)^N\|\boldsymbol{f}\|_{TH^{-1/2}({\rm div}, \Gamma_R)}\Bigg)
\|\boldsymbol{\xi}\|_{\boldsymbol{H}({\rm curl},\Omega)}\\
&&+\delta\|\nabla\times\boldsymbol{\xi}\|^2_{\boldsymbol{L}^2(\Omega)}+C(\delta)\|\boldsymbol{\xi}\|^2_{
	\boldsymbol{L}^2(\Omega)}+C\|\boldsymbol{\xi}\|_{\boldsymbol{L}^2(\Omega)}^2,
\end{eqnarray*}
which gives after taking $\delta=1/2$ that
\begin{equation}\label{theorem_1}
\|\boldsymbol{\xi}\|_{H({\rm curl},\Omega)}^2
\lesssim \Bigg(\Bigg(\sum_{T\in\mathcal{M}_h}\eta_T^2\Bigg)^{1/2}
+\left(\frac{R'}{R}\right)^N\|\boldsymbol{f}\|_{TH^{-1/2}({\rm div}, \,\Gamma_R)}\Bigg)
\|\boldsymbol{\xi}\|_{\boldsymbol{H}({\rm curl},\Omega)}+\|\boldsymbol{\xi}\|_{\boldsymbol{L}^2(\Omega)}^2.
\end{equation}
The proof is completed by substituting \eqref{L2bounds} into \eqref{theorem_1}.
\end{proof}
  
\section{Numerical experiments} \label{sec:num}
 
In this section, we present two numerical examples to demonstrate the efficiency of the adaptive finite element DtN method.
It is shown in Theorem \ref{thm} that the a posteriori error estimator consists two parts: the finite element
discretization error $\epsilon_h$ and the DtN truncation error $\epsilon_N$ which depends on the truncation number $N$. Explicitly
\begin{equation*}\label{epsilonN}
\epsilon_h = \Bigg(\sum\limits_{T\in\mathcal M_h} \eta^2_{T}\Bigg)^{1/2}, \quad \epsilon_N =\left(\frac{R'}{R}\right)^N\|\boldsymbol{f}\|_{TH^{-1/2}({\rm div}, \Gamma_R)}. 
\end{equation*}
The algorithm of the adaptive finite element DtN method is summarized in Table 1.

\begin{table}
\caption{The adaptive finite element DtN method for the electromagnetic scattering
problem.}
\hrule \hrule
\vspace{0.8ex}
\begin{enumerate}		
\item Given the tolerance $\epsilon>0, \theta\in(0,1)$;
\item Fix the computational domain $\Omega=B_R\setminus \overline{D}$ by
choosing the radius $R$;
\item Choose $ \hat{R}$ and $N$ such that $\epsilon_N\leq 10^{-8}$;
\item Construct an initial triangulation $\mathcal M_h$ over $\Omega$ and
compute error estimators;
\item While $\epsilon_h>\epsilon$ do
\item \qquad Refine the mesh $\mathcal M_h$ according to the strategy:
\[
\text{if } \eta_{\hat{T}}>\theta \max\limits_{T\in \mathcal M_h}
\eta_{T}, \text{ then refine the element } \hat{T}\in M_h;
\]
\item \qquad Denote refined mesh still by $\mathcal M_h$, solve the discrete
problem \eqref{vNhform} on the new mesh $\mathcal M_h$;
\item \qquad Compute the corresponding error estimators;
\item End while.
\end{enumerate}
\vspace{0.8ex}
\hrule\hrule
\end{table}

Our implementation is based on the parallel hierarchical grid (PHG) \cite{phg}, which is a toolbox for developing parallel adaptive finite element methods on unstructured tetrahedral meshes. The linear system resulted from finite element discretization is solved by the MUMPS direct solver \cite{mumps}. 
 
\begin{example}\label{ex1}
	Let the obstacle $D=B_{0.1}$ be a ball centered at the origin with radius $0.1$.  
	The Dirichlet boundary condition on $\Gamma_D$ is set by the exact solution
	\begin{equation*}
	\boldsymbol{E}(\boldsymbol{x})=\boldsymbol{G}(\boldsymbol{x})+ k^{-2}\nabla\nabla\cdot\boldsymbol{G}(\boldsymbol{x}),
	\end{equation*}
	where the wavenumber $k=2$ and
	\begin{equation*} 
	\boldsymbol{G}=(0,0,\Phi), \quad \Phi(\boldsymbol x, \boldsymbol y)=\frac{e^{{\rm i} k|\boldsymbol{x}-\boldsymbol{y}|}}{4\pi|\boldsymbol{x}-\boldsymbol{y}|}, \quad \boldsymbol{y}=(0,0,0),
    \end{equation*}
	i.e., the point source is located at $\boldsymbol{y}=(0,0,0)^\top$. The truncated computational domain is defined by 
	$B_{0.5}$, which is a ball centered at the origin with radius $0.5$. 
\end{example}

The surface plots of the amplitude of the field $\boldsymbol{E}_h^N$ are shown in Figure \ref{ex1:solu}.  
Figure \ref{ex1:err} shows the curves of ${\rm log} \|\boldsymbol{E}-\boldsymbol{E}_h^N\|$
versus  ${\rm log} N_k$ for both the a priori and the a posteriori error estimates, where $N_k$ is the total
number of degrees of freedom (DoFs) of the mesh. It indicates that the meshes and the associated numerical
complexity are quasi-optimal, i.e., ${\rm log} \|\boldsymbol{E}-\boldsymbol{E}_h^N\|=O(N_k^{-1/3})$ holds asymptotically.

\begin{figure}
	\centering
	\includegraphics[width=0.4\textwidth]{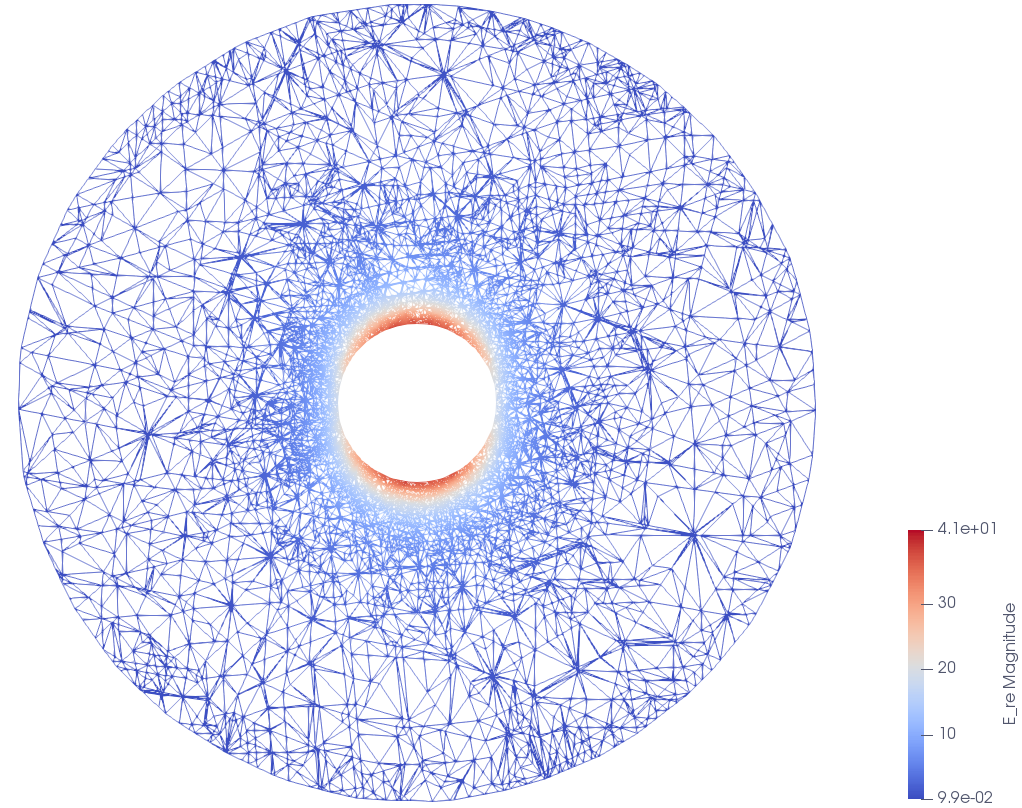}
	\includegraphics[width=0.4\textwidth]{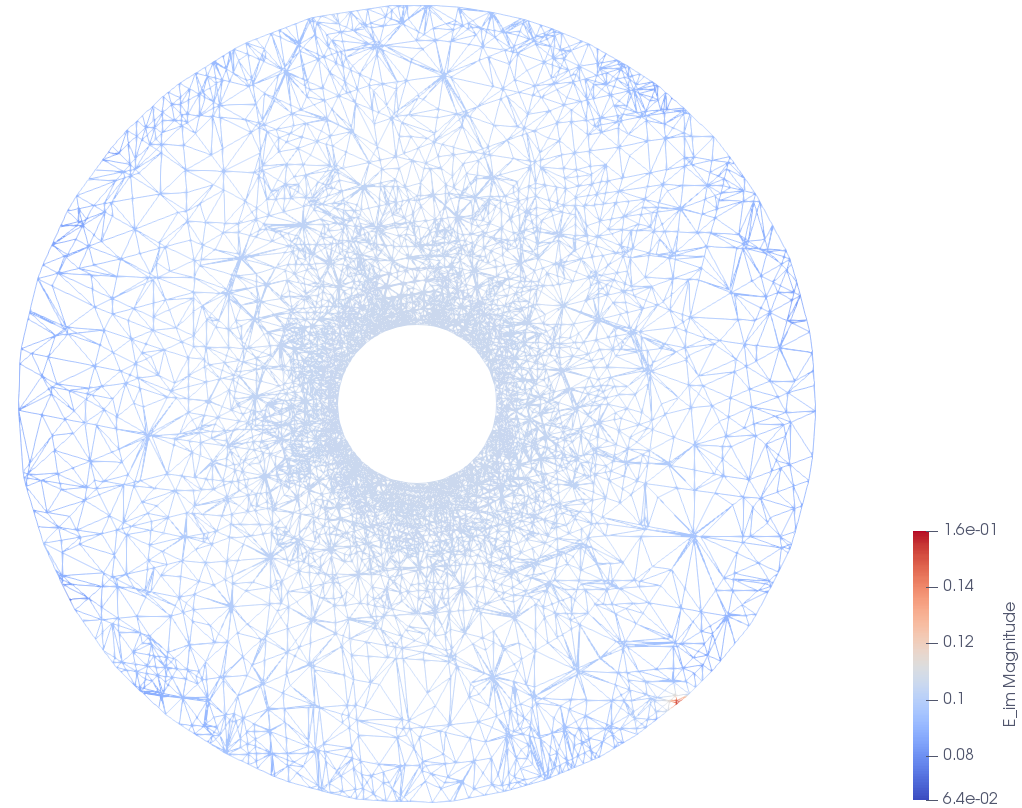}
	\caption{Example~\ref{ex1}: The amplitude of the real part and the imaginary part of the solution $\boldsymbol{E}^N_h$ on the plane $\{\boldsymbol{x}\in\mathbb{R}^3: \,x_1=0\}$.}\label{ex1:solu}
\end{figure}

\begin{figure}
	\centering
	\includegraphics[width=0.4\textwidth]{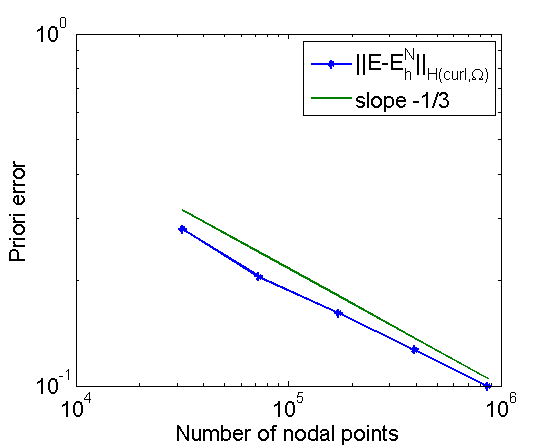}
	\includegraphics[width=0.4\textwidth]{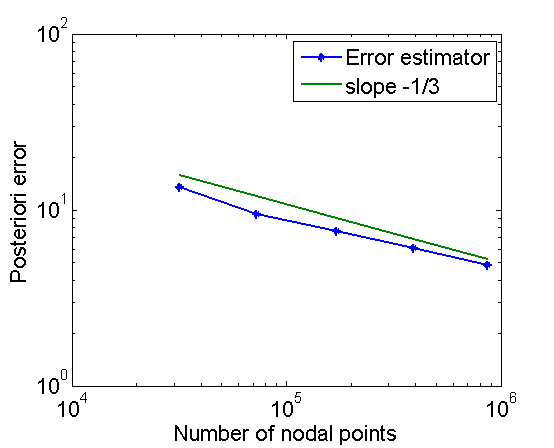}
	\caption{Example~\ref{ex1}: Quasi-optimality of the a priori error (left) and the a posteriori error (right) 
		estimates.} 
	\label{ex1:err}
\end{figure}

\begin{example}\label{ex2}
	This example concerns the scattering of the incident plane wave
\begin{equation*}
	\boldsymbol{E}^{\rm inc}=\boldsymbol{p} e^{{\rm i} k \boldsymbol{q}\cdot\boldsymbol{x}} = e^{-{\rm i} k z}(1,0,0)^\top.
\end{equation*}
	Let the obstacle $D$ be a U-shaped domain, as shown in Figure \ref{ex2:solu}. The Dirichlet boundary condition on $\Gamma_D$ is set by $\boldsymbol{E}=-\boldsymbol{E}^{\rm inc}$.
\end{example}

\begin{figure}
	\centering
	\includegraphics[width=0.4\textwidth]{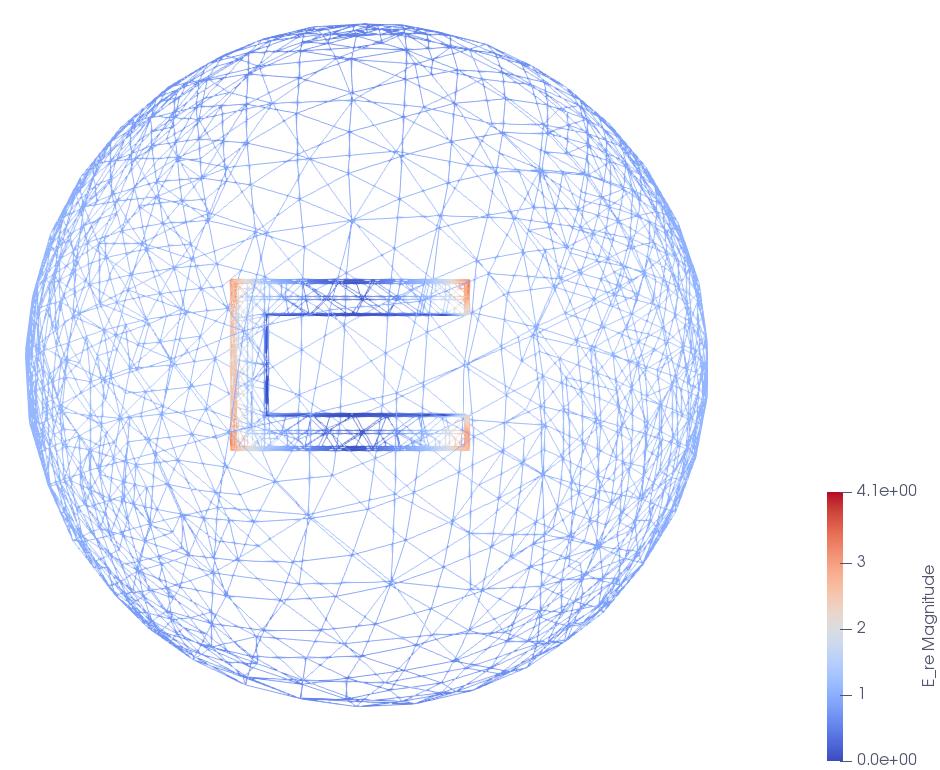}
	\includegraphics[width=0.4\textwidth]{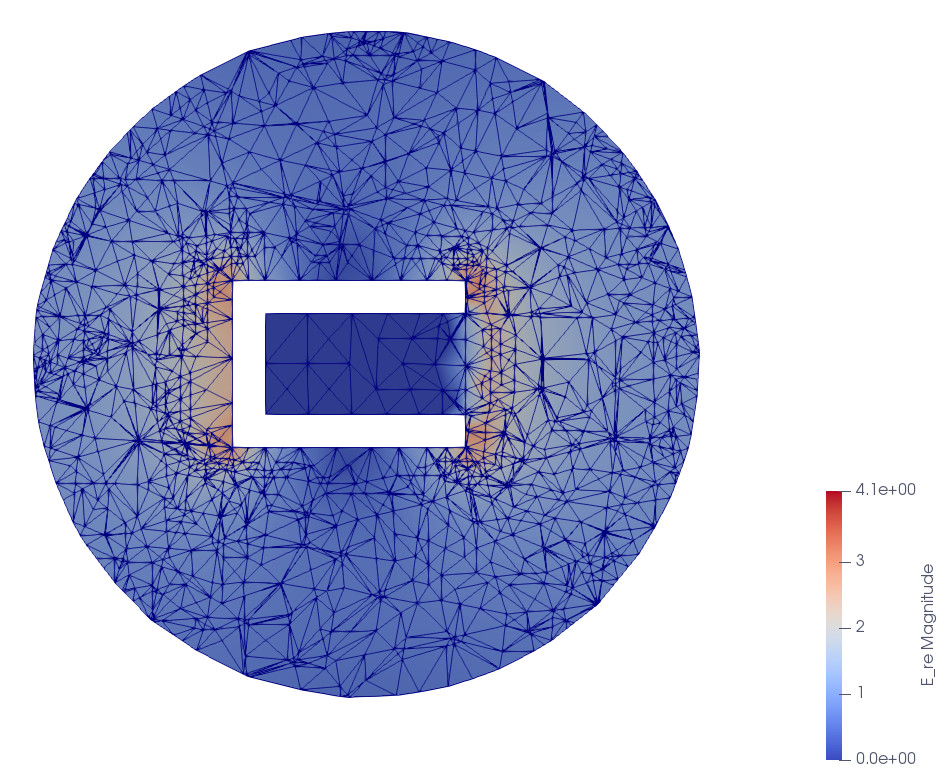}
	\caption{Example~\ref{ex2}: 
		The mesh of the computational domain (left). The amplitude of the real part of the solution $\boldsymbol{E}^N_h$ on the plane $\{\boldsymbol{x}\in\mathbb{R}^3: \,x_2=0\}$ (right).}\label{ex2:solu}
\end{figure}

\begin{figure}
	\centering
	\includegraphics[width=0.4\textwidth]{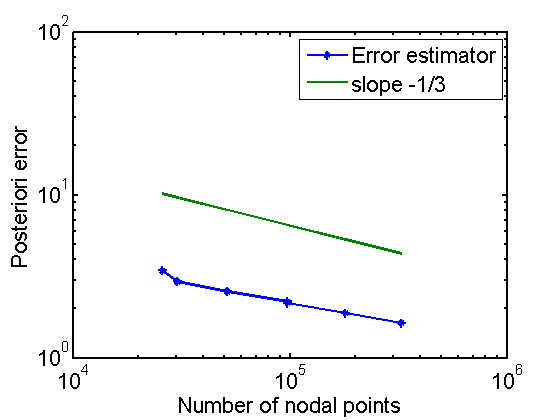}
	\caption{Example~\ref{ex2}: Quasi-optimality of  the a posteriori error  
		estimates.}
	\label{ex2:err}
\end{figure}

The surface plots of the amplitude of the field $\boldsymbol{E}_h^N$ are shown in Figure \ref{ex2:solu}.  
Figure \ref{ex2:err} shows the curves of ${\rm log} \|\boldsymbol{E}-\boldsymbol{E}_h^N\|$
versus  ${\rm log}  N_k$ for the a posteriori error estimate, where $N_k$ is the total
number of DoFs of the mesh. It is clear to note that the meshes and the associated numerical
complexity are quasi-optimal, i.e., ${\rm log} \|\boldsymbol{E}-\boldsymbol{E}_h^N\|=O(N_k^{-1/3})$ is valid asymptotically.

\section{Conclusion}\label{section: conclusion}

In this paper, we have presented an adaptive finite element DtN method for the electromagnetic scattering problem by bounded obstacles in three dimensions. The a posteriori error estimate for the finite element DtN solution is deduced. The posteriori error estimate takes
into account the finite element discretization error and DtN operator truncation error. The latter is shown to decay exponentially with respect to the truncation number. Based on the a posteriori error estimate, an adaptive finite element method is developed. Numerical results show that the proposed method is effective to solve the electromagnetic scattering problem. 

\appendix

\section{Spherical harmonic functions}\label{Appendix:SphericalF}

The spherical coordinates $(\rho,\theta,\varphi)$ are related to the Cartesian
coordinates $\boldsymbol{x}=(x_1,x_2,x_3)$ by $x_1=\rho\sin\theta\cos\varphi$,
$x_2=\rho\sin\theta\sin\varphi$, $x_3=\rho\cos\theta$, where $\theta\in [0,
\pi], \varphi\in [0, 2\pi]$ are the Euler angles of $\boldsymbol{x}$ and $\rho=|\boldsymbol
x|$. The local orthonormal basis $\{\boldsymbol{e}_{\rho}, \boldsymbol
e_{\theta}, \boldsymbol e_{\varphi}\}$ is given by 
\[
\begin{cases}
\boldsymbol{e}_{\rho}=(\sin\theta\cos\varphi,\sin\theta\sin\varphi,\cos\theta),\\
\boldsymbol
e_{\theta}=(\cos\theta\cos\varphi,\cos\theta\sin\varphi,-\sin\theta) ,\\
\boldsymbol e_{\varphi}=(-\sin\varphi,\cos\varphi,0). 
\end{cases}
\]

Denote by $\Gamma=\{\boldsymbol x\in\mathbb R^3: |\boldsymbol x|=1\}$ and
$\Gamma_R=\{\boldsymbol x\in\mathbb R^3: |\boldsymbol x|=R\}$ the unit
sphere and the sphere with radius $R$, respectively. Let $\{Y_n^m(\theta,
\varphi): |m|\leq n, n=0, 1, 2, \dots\}$ be the orthonormal sequence of
spherical harmonics of order $n$ on the unit sphere $\Gamma$. Explicitly, we
have
\begin{equation*}
Y_{n}^{m}(\theta,\varphi)
=\sqrt{\frac{(2n+1)(n-|m|)!}{4\pi(n+|m|)!}}P_n^{|m|}(\cos\theta)
e^{{\rm i}m\varphi}, 
\end{equation*}
where $P_n^m(t), 0\leq m\leq n, -1\leq t\leq 1$ are the associated Legendre
functions and are defined by 
\[
P_n^m(t)=(1-t^2)^{\frac{m}{2}}\frac{{\rm d}^m}{{\rm d}t^m}P_n(t). 
\]
Here $P_n$ is the Legendre polynomial of degree $n$. Define a sequence of
rescaled harmonics of order $n$:
\begin{equation*}
X_n^m(\theta,\varphi)=\frac{1}{R}Y_n^m(\theta,\varphi).
\end{equation*}
It can be easily verified that $\{X_n^m(\theta,\varphi): |m|\leq n, n=0, 1, 2,
\dots\}$ form a complete orthonormal system in $L^2(\Gamma_R)$, which is the functional
space of complex square integrable functions on the sphere $\Gamma_R$. 

\section{Surface differential operators and basis functions}\label{Appendix:SurOpe}

For a smooth scalar function $\phi$ defined on $\Gamma_R$, let
\begin{equation*}
\nabla_{\Gamma} \phi=\frac{\partial \phi}{\partial\theta}\boldsymbol{e}_{\theta}+\frac{1}{\sin\theta}\frac{\partial \phi}{\partial\varphi}\boldsymbol{e}_{\varphi}
\end{equation*}
be the surface gradient on $\Gamma_R$. The surface vector curl is defined by
\begin{equation*}
\textbf{curl}_{\Gamma}\phi=\nabla_{\Gamma} \phi\times \boldsymbol{e}_{\rho}.
\end{equation*}
For a smooth tangent vector function $\boldsymbol{\phi}$ to $\Gamma_R$,
it can be represented by its coordinates in the local orthonormal basis:
\[
\bold{\phi}=\phi_{\theta}\boldsymbol{e}_{\theta}+\phi_{\varphi}\boldsymbol{e}_{\varphi},
\]
where
\[
\phi_{\theta}=\boldsymbol{\phi}\cdot\boldsymbol{e}_{\theta},\quad \phi_{\varphi}=\boldsymbol{\phi}\cdot\boldsymbol{e}_{\varphi}.
\]
The surface divergence and the surface scalar curl can be defined as
\begin{eqnarray*}
	{\rm div}_{\Gamma}\boldsymbol{\phi}=\frac{1}{\sin\theta}
		\left[\frac{\partial}{\partial\theta}(\phi_{\theta}\sin\theta)
			+\frac{\partial \phi_{\varphi}}{\partial\varphi}\right], \qquad
	{\rm curl}_{\Gamma}\boldsymbol{\phi}=\frac{1}{\sin\theta}
		\left[\frac{\partial}{\partial\theta}(\phi_{\varphi}\sin\theta)
			-\frac{\partial \phi_{\theta}}{\partial\varphi}\right].
\end{eqnarray*}

Following \cite[Theorem 6.23]{CK98}, we introduce an orthonormal basis for
$TL(\Gamma_R)$: 
\begin{equation}\label{Unm}
\boldsymbol{U}_n^m(\theta,\varphi)=
\frac{1}{\sqrt{n(n+1)}}\bold{\nabla}_{\Gamma}X_n^m(\theta,\varphi)
\end{equation}
and
\begin{equation}\label{Vnm}
\boldsymbol{V}_n^m(\theta,\varphi)=\boldsymbol{e}_{\rho}\times\boldsymbol{U}_n^m=
-\frac{1}{\sqrt{n(n+1)}}\boldsymbol{\rm curl}_{\Gamma}X_n^m(\theta,\varphi)
\end{equation}
for $|m|\leq n, n=0,1,2,\dots$. 

\section{Identities of differential operators}\label{Appendix:VectorO}

Let $f$ be a smooth function. It can be verified that the curl operator satisfies
\begin{equation}\label{curlUVX}
\left\{
\begin{aligned}
& \nabla\times\left(f(\rho) \boldsymbol{U}_n^m\right)=\frac{1}{\rho}\frac{\partial}{\partial\rho}\left(\rho f(\rho)\right) \boldsymbol{V}_n^m, \\
& \nabla\times\left(f(\rho) \boldsymbol{V}_n^m\right)=-\frac{1}{\rho}\frac{\partial}{\partial\rho}\left(\rho f(\rho)\right)\boldsymbol{U}_n^m-\frac{\sqrt{n(n+1)}}{\rho}f(\rho) X_n^m \boldsymbol{e}_{\rho}, \\
& \nabla\times\left(f(\rho) X_n^m \boldsymbol{e}_{\rho}\right)=-\frac{\sqrt{n(n+1)}}{\rho}f(\rho) \boldsymbol{V}_n^m.
\end{aligned}
\right.
\end{equation}
Moreover, we may show from \eqref{curlUVX} that
\begin{equation}\label{3Dsurface_curlfe}
\left\{
\begin{aligned}
	& \left(\nabla\times\left(f(\rho) \boldsymbol{U}_n^m\right)\right)\times \boldsymbol{e}_{\rho}
		=\frac{1}{\rho}\frac{\partial}{\partial\rho}\left(\rho f(\rho)\right) \boldsymbol{U}_n^m,\\
	& \left(\nabla\times\left(f(\rho) \boldsymbol{V}_n^m\right)\right)\times \boldsymbol{e}_{\rho}
		=\frac{1}{\rho}\frac{\partial}{\partial\rho}\left(\rho f(\rho)\right)\boldsymbol{V}_n^m,\\
	& \left(\nabla\times\left(f(\rho) X_n^m \boldsymbol{e}_{\rho}\right)\right)\times \boldsymbol{e}_{\rho}
		=-\frac{\sqrt{n(n+1)}}{\rho}f(\rho) \boldsymbol{U}_n^m. 
\end{aligned}	
\right.	
\end{equation}
Taking the curl on both sides of \eqref{curlUVX}, we have 
\begin{equation}\label{curlcurlUVX}
\left\{
\begin{aligned}
	& \nabla\times\left(\nabla\times\left(f(\rho) \boldsymbol{U}_n^m\right)\right)=
		-\frac{1}{\rho}\frac{\partial^2}{\partial \rho^2}\left(\rho f(\rho)\right)\boldsymbol{U}_n^m
		-\frac{\sqrt{n(n+1)}}{\rho^2}\frac{\partial}{\partial\rho}\left(\rho f(\rho)\right)X_n^m \boldsymbol{e}_{\rho},\\
	& \nabla\times\left(\nabla\times\left(f(\rho) \boldsymbol{V}_n^m\right)\right)=
		\left[-\frac{1}{\rho}\frac{\partial^2}{\partial \rho^2}\left(\rho f(\rho)\right)
		+\frac{n(n+1)}{\rho^2} f(\rho)\right] \boldsymbol{V}_n^m, \\
	& \nabla\times\left(\nabla\times\left(f(\rho) X_n^m \boldsymbol{e}_{\rho}\right)\right)=
		\frac{\sqrt{n(n+1)}}{\rho}\frac{\partial}{\partial\rho} f(\rho) \boldsymbol{U}_n^m 
		+\frac{n(n+1)}{\rho^2} f(\rho) X_n^m \boldsymbol{e}_{\rho}.
\end{aligned}
\right.
\end{equation}
The divergence operator satisfies 
\begin{equation}\label{divUVX}
\left\{
\begin{aligned}
	& \nabla\cdot\left(f(\rho)\boldsymbol{U}_n^{m}\right)=-f(\rho)\sqrt{(n+1)n}\,\frac{1}{\rho}X_{n}^{m},\\
	&\nabla\cdot\left(f(\rho)\boldsymbol{V}_n^m\right)=0, \\
	&\nabla\cdot\left(f(\rho)X_n^{m} \boldsymbol{e}_{\rho}\right)
		=\frac{1}{\rho^2}\frac{\partial}{\partial \rho}\left(\rho^2 f(\rho)\right)X_n^m.
\end{aligned}
\right.
\end{equation}

The following result can be easily obtained from \eqref{divUVX}. 

\begin{lemma}\label{3Dsurface_divfree}
Given any smooth vector function 
\[
\boldsymbol{v}_n^{m}=v_{1n}^m(\rho) \boldsymbol{U}_n^m+v_{2n}^m(\rho) \boldsymbol{V}_n^m+v_{3n}^m(\rho) X_n^m \boldsymbol{e}_{\rho},
\]
if $\nabla\cdot\boldsymbol{v}_n^m=0$, then its coefficients satisfy the following equation
\[
\frac{\partial}{\partial\rho}\left(\rho^2 v_{3n}^m(\rho)\right)=\sqrt{n(n+1)}\rho v_{1n}^m(\rho).
\]
\end{lemma}

\end{document}